\documentclass[11pt]{article}
\usepackage{amsfonts}
\usepackage{latexsym,amssymb,amsmath,graphics,cite,arydshln}
\usepackage{tikz}
\usepackage{fancyhdr}
\usepackage{lipsum}
\usepackage{lineno}
\usepackage{color}
\usepackage{rotfloat}

\begin{document}
\newcommand{\qed}{\hphantom{.}\hfill $\Box$\medbreak}
\newcommand{\proof}{\noindent{\bf Proof \ }}

\newtheorem{theorem}{Theorem}[section]
\newtheorem{lemma}[theorem]{Lemma}
\newtheorem{corollary}[theorem]{Corollary}
\newtheorem{remark}[theorem]{Remark}
\newtheorem{example}[theorem]{Example}
\newtheorem{definition}[theorem]{Definition}
\newtheorem{construction}[theorem]{Construction}
\newtheorem{fact}[theorem]{Fact}
\newtheorem{proposition}[theorem]{Proposition}

\begin{center}
{\Large\bf Constructions for optimal Ferrers diagram rank-metric codes\footnote{Supported by NSFC under Grant $11431003$, $11471032$, and Fundamental Research Funds for the Central Universities under Grant $2016$JBM$071$, $2016$JBZ$012$}
}

\vskip12pt

Shuangqing Liu, Yanxun Chang, and Tao Feng\\[2ex] {\footnotesize Department of Mathematics, Beijing Jiaotong University, Beijing 100044, P. R. China}\\
{\footnotesize
16118420@bjtu.edu.cn, yxchang@bjtu.edu.cn, tfeng@bjtu.edu.cn}
\vskip12pt

\end{center}

\vskip12pt

\noindent {\bf Abstract:} Optimal rank-metric codes in Ferrers diagrams can be used to construct good subspace codes. Such codes consist of matrices having zeros at certain fixed positions. This paper generalizes the known constructions for Ferrers diagram rank-metric (FDRM) codes. Via a criterion for linear maximum rank distance (MRD) codes, an explicit construction for a class of systematic MRD codes is presented, which is used to produce new optimal FDRM codes. By exploring subcodes of Gabidulin codes, if each of the rightmost $\delta-1$ columns in Ferrers diagram $\cal F$ has at least $n-r$ dots, where $r$ is taken in a range, then the conditions that an FDRM code in $\cal F$ is optimal are established. The known combining constructions for FDRM code are generalized by introducing the concept of proper combinations of Ferrers diagrams.

\vskip12pt



\section{Introduction}

Network coding, introduced in \cite{acly}, refers to coding at the intermediate nodes when information is multicasted in a network. Often information is modeled as vectors of fixed length over a finite field $\mathbb F_q$, called {\em packets}. To improve the performance of the communication, intermediate nodes should forward random linear $\mathbb F_q$-combinations of the packets they receive. Hence, the vector
space spanned by the packets injected at the source is globally preserved in the network when no error occurs.

This observation led K\"{o}tter and Kschischang \cite{kk} to model network codes as
projective space ${\cal P}_q(n)$, the set of all subspaces of $\mathbb F_q^n$, or
Grassmann space ${\cal G}_q(n,k)$, the set of all subspaces of $\mathbb F_q^n$ having dimension $k$. Subsets of ${\cal P}_q(n)$ are called {\em subspace codes} or {\em projective codes}, while subsets of the Grassmann space are referred to as {\em constant-dimension codes} or {\em Grassmann codes}. The subspace distance $d_S(U,V)={\rm dim}U+{\rm dim}V-2{\rm dim}(U\cap V)$ for all $U,V\in {\cal P}_q(n)$ is used as a distance measure for subspace codes. For more information on constructions and bounds for subspace codes, the interested reader may refer to \cite{es09,es13,ev11,gy,kku,se13,se131,st15,s,tmbr}.

Silva, Kschischang and K\"{o}tter \cite{skk} pointed out that lifted maximum rank distance (MRD) codes can result in almost optimal constant dimension codes, which asymptotically attain the known upper bounds \cite{kk,ev11}, and can be decoded efficiently in the context of random linear network coding.

To obtain optimal constant dimension codes, Etzion and Silberstein \cite{es09} presented a simple but effective construction, named the multilevel construction, which generalizes the lifted MRD codes construction by introducing a new family of rank-metric codes, namely, Ferrers diagram rank-metric codes. Furthermore, Etzion, Gorla, Ravagnani and Wachter-Zeh \cite{egrw} systematically investigated Ferrers diagram rank-metric codes and established four constructions to obtain optimal codes.

This paper continues the work in \cite{egrw}. In Section 2, we give a brief introduction of Ferrers diagram rank-metric codes, and review most of known constructions in the literature.

Via a criterion for linear MRD codes presented in \cite{tm}, we give an explicit construction for a class of systematic MRD codes in Section 3.1, which can be used to produce optimal Ferrers diagram rank-metric codes (see Construction \ref{con:from sys MRD}). In Section 3.2, we generalize Construction 2 in \cite{egrw} by exploring subcodes of Gabidulin codes. Construction 2 in \cite{egrw} requires that each of the rightmost $\delta-1$ columns in Ferrers diagram $\cal F$ has at least $n-1$ dots. We relax the condition $n-1$ to $n-r$, where $r$ is taken in a range (see Theorem \ref{thm:subcodes from Gab}).

In Section 4, by introducing the concept of proper combinations of Ferrers diagrams, we generalize Theorem 9 in \cite{egrw}. Our constructions are essentially to combine small Ferrers diagram rank-metric codes to a bigger one more flexibly (see Constructions \ref{con:com-1}, \ref{con:com-2} and \ref{con:com-3}).

\section{Preliminaries}

Let $q$ be a prime power, $\mathbb F_q$ be the finite field of order $q$, and $\mathbb F_{q^m}$ be its extension field of order $q^m$. We use $\mathbb F^{m \times n}_q$ to denote the set of all $m \times n$ matrices over $\mathbb F_q$, and $\mathbb F^n_{q^m}$ to denote the set of all row vectors of length $n$ over $\mathbb F_{q^m}$. The rank of a matrix $\textbf{A} \in \mathbb F^{m \times n}_q$ is denoted by rank$(\textbf{A})$. The rows and columns of an $m\times n$ matrix will be indexed by $0,1,\ldots, m-1$ and $0,1,\ldots, n-1$, respectively. Let $[n]$ denote $\{0,1,\ldots,n-1\}$ and $(i,j)$ denote the cell in the $i$-th row and the $j$-th column of an $m\times n$ matrix, where $i\in [n]$ and $j\in [m]$. Write $\textbf{I}_{s}$ as the $s\times s$ identity matrix.

\subsection{Rank-metric codes}

The set $\mathbb F^{m \times n}_q$ is an $\mathbb F_q$-vector space. The {\em rank distance} on $\mathbb F^{m \times n}_q$ is defined by
\begin{center}
  $d_R(\textbf{A}, \textbf{B})={\rm rank}(\textbf{A}-\textbf{B})~{\rm for}~\textbf{A}, \textbf{B} \in \mathbb F^{m \times n}_q$.
\end{center}
An $[m \times n, k, \delta]_q$ {\em rank-metric code} $\mathcal C$ is a $k$-dimensional $\mathbb F_q$-linear subspace of $\mathbb F^{m \times n}_q$ with {\em minimum rank distance} $$\delta=\underset{\textbf{A,B} \in \mathcal C, \textbf{A}\neq \textbf{B}}{\min}\{d_R(\textbf{A},\textbf{B})\}.$$
Clearly $$\delta=\underset{\textbf{A} \in \mathcal C, \textbf{A}\neq \textbf{0}}{\min}\{{\rm rank}(\textbf{A})\}.$$
The Singleton-like upper bound for rank-metric codes implies that $$k \leq {\rm max}\{m,n\}({\rm min}\{m,n\}- \delta +1)$$
holds for any $[m \times n, k, \delta]_q$ code. When the equality holds, $\cal C$ is called a {\em linear maximum rank distance code}, denoted by an MRD$[m \times n, \delta]_q$ code. Linear MRD codes exists for all feasible parameters (cf. \cite{d,g,r}).

\subsection{Ferrers diagram rank-metric codes}

Given positive integers $m$ and $n$, an $m \times n$ Ferrers diagram $\mathcal F$ is an $m \times n$ array of dots and empty cells such that all dots are shifted to the right of the diagram, the number of dots in each row is less than or equal to the number of dots in the previous row, and the first row has $n$ dots and the rightmost column has $m$ dots. The number of dots in $\mathcal F$ is denoted by $|\mathcal F|$.

\begin{example}\label{1}
~
\begin{center}
 $\mathcal{F}=\begin{array}{cccc}
                \bullet & \bullet & \bullet & \bullet \\
                \bullet & \bullet & \bullet & \bullet \\
                 & \bullet & \bullet & \bullet \\
                 &  & \bullet & \bullet \\
                 &  &  & \bullet
              \end{array}
 $
\end{center}
is a $5\times 4$ Ferrers diagram and $|\mathcal F|=14$.
\end{example}

Sometimes it is convenient to state Ferrers diagrams by using the set-theoretical language (cf. \cite{b,gr}).
Given positive integers $m$ and $n$, an $m \times n$ Ferrers diagram $\mathcal F$ is a subset of $[m]\times [n]$ satisfying that (1) if $(i, j) \in \mathcal F$ and $i\geq 1$, then $(i-1, j) \in \mathcal F$; (2) if $(i, j) \in \mathcal F$ and $j\leq n-2$, then $(i, j+1) \in \mathcal F$. In the sequel, these two definitions will be both used, depending on what is more convenient in the context.

Motivated by the multilevel construction from \cite{es09}, some research work has been done on constructing good or even optimal rank-metric codes in Ferrers diagrams \cite{egrw,gr,st15,tr}. For a given $m \times n$ Ferrers diagram $\mathcal F$, an $[\mathcal F, k, \delta]_q$ {\em Ferrers diagram rank-metric} (FDRM) {\em code}, briefly an $[\mathcal F, k, \delta]_q$ code, is an $[m \times n, k, \delta]_q$ rank-metric code in which for each $m \times n$ matrix, all entries not in $\mathcal F$ are zero. If $\mathcal F$ is a {\em full} $m\times n$ diagram with $mn$ dots, then its corresponding FDRM code is just a classical rank-metric code.

Etzion and Silberstein \cite{es09} established a Singleton-like upper bound on FDRM codes.

\begin{lemma} {\rm (Theorem 1 in  \cite{es09})} \label{lem:upper bound}
Let $\delta$ be a positive integer. Let $\cal F$ be a Ferrers diagram and ${\cal C}_{\cal F}$ be any $[\mathcal F, k, \delta]_q$ code. Then $k\leq \min_{i\in[\delta]}v_i$, where $v_i$ is the number of dots in $\cal F$ which are not contained in the first $i$ rows and the rightmost $\delta-1-i$ columns.
\end{lemma}

An FDRM code which attains the upper bound in Lemma \ref{lem:upper bound} is called {\em optimal}. An MRD$[m \times n, \delta]_q$ code with $m\geq n$ is an optimal $[\mathcal F, m(n-\delta+1), \delta]_q$ code, where $\cal F$ is a full $m\times n$ diagram. So far all known FDRM codes over $\mathbb F_q$ with the largest possible dimension are optimal.

We remark that the upper bound still holds for FDRM codes defined on any field, and especially, for algebraically closed fields the bound cannot be attained (see Theorem 13 and Proposition 17 in \cite{gr}). This paper focuses only on finite fields since they are used for forming subspace codes.

For a Ferrers diagram $\mathcal F$ of size $m\times n$, one can transpose it to obtain a Ferrers diagram ${\cal F}^{t}$ of size $n\times m$. Thus if there exists an $[\mathcal F, k, \delta]_q$ code, then so does an $[{\mathcal F}^t, k, \delta]_q$ code. Without loss of generality, we always assume that $m\geq n$.

We denote by $\gamma_i$, $i\in [n]$, the number of dots in the $i$-th column of $\mathcal F$, and by $\rho_i$, $i \in [m]$, the number of dots in the $i$-th row of $\mathcal F$.

\subsection{Known constructions for FDRM codes}

This section summarizes known main constructions for FDRM codes, which come from \cite{al,egrw,es09,gr,zg}. We shall use or generalize them later.

\subsubsection{Exploration of subcodes of MRD codes}

Etzion and Silberstein \cite{es09} introduced the concept of FDRM codes. They established the existence of optimal $[\mathcal F, k, \delta]_q$ codes whenever $\cal F$ is an $m\times n$ ($m\geq n$) Ferrers diagram and each of its rightmost $\delta-1$ columns has at least $m$ dots. The proof is based on the use of $q$-cyclic MRD codes. A better result is provided in \cite{egrw} with a simple proof by means of shortening systematic MRD codes (see also Theorem 23 in \cite{gr} and Corollary 3.3 in \cite{al}).

\begin{theorem} {\rm (Theorem 3 in \cite{egrw})} \label{thm:shortening}
Assume $\mathcal F$ is an $m \times n$ Ferrers diagram and each of the rightmost $\delta-1$ columns of $\mathcal F$ has at least $n$ dots. Then there exists an optimal $[\mathcal F, k, \delta]_q$ code for any prime power $q$, where $k=\sum_{i=0}^{n-\delta} \gamma_i$.
\end{theorem}

As a straightforward corollary, Etzion and Silberstein pointed out the following fact.

\begin{corollary} {\rm \cite{es09}} \label{cor:delta=1,2}
Let $\delta\in\{1,2\}$. There exists an optimal $[\mathcal F, k, \delta]_q$ code for any Ferrers diagram $\cal F$ and any prime power $q$, where $k=\sum_{i=0}^{n-\delta} \gamma_i$.
\end{corollary}

To relax the restriction on $\cal F$ in Theorem \ref{thm:shortening}, the idea of exploring subcodes of MRD codes was introduced to construct FDRM codes in \cite{egrw}, and developed in \cite{al,zg} recently.

\begin{theorem}{\rm (Theorem 8 in \cite{egrw})} \label{thm:from subcodes}
Assume $\mathcal F$ is an $m \times n$ Ferrers diagram and $m\geq n$. Let $2\leq \delta \leq n-1$. If each of the rightmost $\delta-1$ columns in $\cal F$ has at least $n-1$ dots, then there exists an $[\mathcal F, k, \delta]_q$ code for any prime power $q$, where $k=\min\{m-n+1, \gamma_0\} + \sum_{i=1}^{n-\delta} \gamma_i$. When $\gamma_0\leq m-n+1$, the resulting FDRM code is optimal.
\end{theorem}

We shall generalize Theorem \ref{thm:from subcodes} to Theorem \ref{thm:subcodes from Gab}, where it is required that each of the rightmost $\delta-1$ columns in $\cal F$ has at least $n-r$ dots for any positive integer $r$ satisfying $r+1\leq \delta\leq n-r$.

\begin{theorem}{\rm (Theorem 3.6 in \cite{al})} \label{thm:from Max}
Assume $\mathcal F$ is an $m\times n$ Ferrers diagram and $m\geq n$. Let $2\leq \delta\leq n$ and $l=n-\delta+1$. Set $\varepsilon=\sum_{t=l}^{n-1}(m-\gamma_t)$, that is, $\varepsilon$ is the number of dots missing in the rightmost $\delta-1$ columns of $\mathcal F$. If $\gamma_s\leq \gamma_l-\varepsilon(l-s)$ for every $s\in\{0,1,\ldots,l-1\}$, then there exists an optimal $[\mathcal F,\sum_{i=0}^{l-1}\gamma_i,\delta]_q$ code.
\end{theorem}

Theorem \ref{thm:from Max} implies Theorem \ref{thm:shortening} when $\varepsilon=0$. When $\varepsilon\neq 0$, the condition $\gamma_s\leq \gamma_l-\varepsilon(l-s)$ for $s\in\{0,1,\ldots,l-1\}$ means that the numbers of dots in the first $l$ columns are restricted in an arithmetic progression with step size $\varepsilon$.

\begin{theorem}{\rm (Theorem 3.6 in \cite{zg})} \label{the:from sys MRD-new1}
Let $l$ be a positive integer. Let $1=t_0<t_1<t_2<\cdots<t_l$ be integers such that $t_1\mid t_2\mid \cdots\mid t_l$.  Let $n$ and $\delta$ be positive integers satisfying $t_{l-1}<n-1\leq t_l$ and $n-t_1+1< \delta\leq n-1$. Let $\mathcal F$ be an $m\times n$ Ferrers diagram satisfying
\begin{itemize}
\item[$(1)$] $\gamma_{n-\delta}\leq t_1$,
\item[$(2)$] $\gamma_{n-\delta+1}\geq t_1$,
\item[$(3)$] $\gamma_{t_i}\geq t_{i+1}$ for $1\leq i \leq l-1$,
\item[$(4)$] $\gamma_{n-1}\geq t_{l}+\gamma_0$,
\end{itemize}
 Then there exists an optimal $[\mathcal F, \sum_{i=0}^{n-\delta} \gamma_i, \delta]_q$ code for any prime power $q$.
\end{theorem}

When $l=1$, Theorem \ref{the:from sys MRD-new1} together with Theorem \ref{thm:shortening} yields Theorem \ref{thm:from subcodes} (note that to remove the condition $\gamma_{n-\delta}\leq t_1$, Theorem \ref{thm:shortening} is needed).

\subsubsection{Use of MDS codes}

A construction for FDRM codes based on maximum distance separable (MDS) codes is presented in \cite{egrw}. It is known that an $[n,n-d+1,d]_{q}$ MDS code exists for any $q\geq n-1$ or $d\in \{1,2,n\}$ (see \cite{ms}).

A {\em diagonal} of a Ferrers diagram $\mathcal F$ is a consecutive sequence of entries, going upwards diagonally from the rightmost column to either the leftmost column or the first row. Let $D_i$, $i\in [m]$, denote the $i$-th diagonal in $\mathcal F$, where $i$ counts the diagonals from the top to the bottom and let $\theta_i$ denote the number of dots on $D_i$ in $\mathcal F$.

\begin{example}
For the Ferrers diagram in Example $\ref{1}$, its five diagonals are:
\begin{center}
$D_0=\bullet,~
D_1=\begin{array}{cc}
                    \bullet &  \\
                     & \bullet
                  \end{array},~
D_2=\begin{array}{ccc}
      \bullet &  &  \\
       & \bullet &  \\
       &  & \bullet
    \end{array},~
D_3=\begin{array}{cccc}
      \bullet &  &  &  \\
       & \bullet &  &  \\
       &  & \bullet &  \\
       &  &  & \bullet
    \end{array},~
D_4=\begin{array}{cccc}
      \bullet &  &  &  \\
       & \bullet &  &  \\
       &  & \bullet &  \\
       &  &  & \bullet
    \end{array}.$
\end{center}
\end{example}

\begin{theorem}{\rm (Construction 1 in  \cite{egrw})} \label{thm:from MDS}
Let $\mathcal F$ be an $m \times n$ Ferrers diagram and $\delta$ be an integer such that $0<\delta \leq n$. Let $\theta_{max}=\max_{i \in [m]}{\theta_i}$. Then there exists an $[\mathcal F, k, \delta]_q$ code for any prime power $q\geq \theta_{max}-1$, where $k=\sum_{i=0}^{m-1} {\max\{0, \theta_i-\delta +1\}}$.
\end{theorem}

Applying Theorems \ref{thm:shortening} and \ref{thm:from MDS}, Etzion and Silberstein obtained the following result.

\begin{corollary} {\rm (Theorem 11 in \cite{egrw})} \label{cor:delta=3}
Let $n\geq 3$. There exists an optimal $[\mathcal F, k, 3]_q$ code for any $n\times n$ Ferrers diagram $\cal F$ and any prime power $q$.
\end{corollary}

The disadvantage of Theorem \ref{thm:from MDS} is the requirement of large $q$. For example when $\cal F$ is an $n\times n$ Ferrers diagram with $i+1$ dots in its $i$-th column for $i\in [n]$, by Theorem \ref{thm:from MDS}, there exists an optimal $[\mathcal F, 3, n-1]_q$ code for any prime power $q\geq n-1$. Recently Antrobus and Gluesing-Luerssen showed that such optimal FDRM codes exist for any prime power $q$ via induction on $n$.

\begin{theorem} {\rm (Theorem 5.2 in \cite{al})} \label{thm:upper triangular}
Let $n\geq 3$. Assume $\cal F$ is an $n\times n$ Ferrers diagram with $i+1$ dots in its $i$-th column for $i\in [n]$. Then there exists an optimal $[\mathcal F, 3, n-1]_q$ code for any prime power $q$.
\end{theorem}

However, how to give other constructions for FDRM codes with the same parameters as those obtained from Theorem \ref{thm:from MDS}, but for any prime power $q$, is still an open problem. We shall exhibit three examples in Section 3 (Examples \ref{eg:sys MDS-1}, \ref{eg:sys MDS-2}, \ref{eg:sys MDS-3}) to touch this problem.

\subsubsection{Combination of FDRM codes}

To obtain new FDRM codes based on known ones, \cite{egrw} presented an excellent idea. We shall extend this idea in Section 4.

\begin{theorem}{\rm (Theorem 9 in \cite{egrw})} \label{thm:combine with same dim}
Let $\mathcal F_i$ for $i=1,2$ be an $m_i \times n_i$ Ferrers diagram, and $\mathcal C_i$ be an $[\mathcal F_i, k, \delta_i]_q$ code. Let $\mathcal D$ be an $m_3 \times n_3$ full Ferrers diagram with $m_3n_3$ dots, where $m_3 \geq m_1$ and $n_3 \geq n_2$. Let
\begin{center}
$\mathcal F=\left(
  \begin{array}{cc}
    \mathcal F_1 & \mathcal D \\
      & \mathcal F_2 \\
  \end{array}
\right)$
\end{center}
be an $m \times n$ Ferrers diagram, where $m=m_2+m_3$ and $n=n_1+n_3$. Then there exists an $[\mathcal F, k, \delta_1+\delta_2]_q$ code.
\end{theorem}

The limitation of Theorem \ref{thm:combine with same dim} can be shown in the following lemma.

\begin{lemma}\label{lem:dis combin}
Let $\delta$ be a positive integer. Let $\cal F$ be an $m\times n$ Ferrers diagram satisfying $$ \max_{n-\delta+1\leq i\leq n-1}(\gamma_i-\gamma_{i-1})<v_0= \min_{i\in[\delta]}v_i,$$ where $v_i$ is the number of dots in $\cal F$ which are not contained in the first $i$ rows and the rightmost $\delta-1-i$ columns. Then one cannot apply Theorem $\ref{thm:combine with same dim}$ to construct an optimal $[\mathcal F, v_0, \delta]_q$ code.
\end{lemma}

\proof Assume that an optimal $[\mathcal F, v_0, \delta]_q$ code can be constructed by Theorem \ref{thm:combine with same dim}, where
\begin{center}
$\mathcal F=\left(
  \begin{array}{cc}
    \mathcal F_1 & \mathcal D \\
      & \mathcal F_2 \\
  \end{array}
\right)$
\end{center}
is an $m \times n$ Ferrers diagram, $\mathcal F_j$ is an $m_j\times n_j$ Ferrers diagram, $\mathcal C_j$ is an $[\mathcal F_j, v_0, \delta_j]_q$ code for $j=1,2$, and $\delta=\delta_1+\delta_2$. Let $v_0^{(j)}$ be the number of dots in $\mathcal F_j$ which are not contained in the rightmost $\delta_j-1$ columns.

Consider $\mathcal C_1$. By Lemma \ref{lem:upper bound}, $v_0\leq v_0^{(1)}$, i.e., $\sum_{i=0}^{n-\delta} \gamma_i\leq \sum_{i=0}^{n_1-\delta_1} \gamma_i$, which yields $n-\delta\leq n_1-\delta_1$. Thus $n_1\geq n-\delta+\delta_1$ and $\delta_2=\delta-\delta_1\geq n-n_1$.

Consider $\mathcal C_2$. Since $\delta_2\leq n_2\leq n-n_1$, we have $\delta_2=n_2=n-n_1$. By Lemma \ref{lem:upper bound}, the existence of an $[\mathcal F_2, v_0, n_2]_q$ code implies $v_0$ is no more than the number of dots in $\mathcal F_2$ which are not contained in the rightmost $n_2-1$ columns. Hence, $v_0\leq \gamma_{n_1}-\gamma_{n_1-1}$, which contradicts with the known condition $\max_{n-\delta+1\leq i\leq n-1}(\gamma_i-\gamma_{i-1})<v_0$. \qed

\section{Constructions based on subcodes of MRD codes}

Let \boldmath $\mathbf{\beta}$ \unboldmath = $(\beta_0,  \beta_1, ..., \beta_{m-1})$ be an ordered basis of $\mathbb F_{q^m}$ over $\mathbb F_q$. There is a natural bijective map $\Psi_m$ from $\mathbb F_{q^m}^n$
to $\mathbb F_{q}^{m \times n}$ as follows:
\begin{center}
  $\Psi_m: \mathbb F^n_{q^m} \longrightarrow \mathbb F^{m\times n}_q$
\end{center}
\begin{center}
  $\textbf{a}=(a_0, a_1, \ldots, a_{n-1}) \longmapsto \textbf{A},
  $
\end{center}
where $\textbf{A}=\Psi_m(\textbf{a})\in \mathbb F^{m\times n}_q$ is defined such that
\begin{equation*}
 a_j=\sum_{i=0}^{m-1} A_{i, j}\beta_i
\end{equation*}
for any $j\in [n]$. For $a\in \mathbb F_{q^m}$, $(a)$ is a $1\times 1$ matrix and we simply write $\Psi_m((a))$ as $\Psi_m(a)$. It is readily checked that $\Psi_m$ satisfies linearity, i.e., $\Psi_m(x\textbf{c}_1+y\textbf{c}_2)=x\Psi_m(\textbf{c}_1)+y\Psi_m(\textbf{c}_2)$ for any $x,y\in \mathbb F_{q}$ and $\textbf{c}_1,\textbf{c}_2\in \mathbb F^n_{q^m}$.
The map $\Psi_m$ will be used to facilitate switching between a vector in $\mathbb F_{q^m}$ and its matrix representation over $\mathbb F_q$. In the sequel, we use both representations, depending on what is more convenient in the context and by slight abuse of notation, rank$(\textbf{a})$ denotes rank$({\Psi_m}(\textbf{a}))$.

The following lemma, implicitly shown in Section 5 in \cite{egrw}, is fundamental to construct FDRM codes via subcodes of MRD codes. All theorems in Section 2.3.1 are based on this lemma.

\begin{lemma}{\rm \cite{egrw}}\label{lem:subcode from MRD}
Assume that $m\geq n$. Let $\textbf{G}$ be a generator matrix of a systematic MRD$[m\times n,\delta]_q$ code, i.e., $\textbf{G}$ is of the form $(\textbf{I}_k|\textbf{A})$, where $k=n-\delta+1$. Let $0\leq \lambda_0 \leq \lambda_1\leq \cdots \leq \lambda_{k-1}\leq m$. Let $\textbf{U}=\{(u_{0},\ldots,u_{k-1}) \in \mathbb F_{q^{m}}^k:\Psi_{m}(u_i)=(u_{i,0},\ldots,u_{i,\lambda_{i}-1},0,\ldots,0)^T, u_{i,j}\in \mathbb F_q,i\in [k],j\in[\lambda_i]\}$. Then $\mathcal C=\{\Psi_m(\textbf{c}):\textbf{c}=\textbf{u}\textbf{G},\textbf{u}\in\textbf{U}\}$ is a linear FDRM code with dimension $\sum_{i=0}^{k-1} \lambda_i$ and rank at least $\delta$ over $\mathbb F_q$.
\end{lemma}

\begin{proof}
One can easily verify the linearity and the dimension of the code. Since $\textbf{G}$ is a generator matrix of an MRD$[m \times n, \delta]_{q}$ code $\mathcal C_{M}$, $\mathcal C$ is a subcode of $\mathcal C_{M}$. So the minimum rank distance of the code $\mathcal C$ is $\delta$. \qed
\end{proof}

Lemma \ref{lem:subcode from MRD} doesn't show the Ferrers diagram used explicitly. However, if we could know more about the initial MRD code, then it would be possible to give a complete characterization of $\cal C$.

\begin{remark}\label{rek:subcodes}
Lemma $\ref{lem:subcode from MRD}$ can only be used to construct optimal FDRM codes satisfying $v_0=\sum_{i=0}^{n-\delta}\gamma_i = \min_{i\in[\delta]}v_i$ $($cf. Lemma $\ref{lem:upper bound})$, where $v_i$ is the number of dots in $\cal F$ which are not contained in the first $i$ rows and the rightmost $\delta-1-i$ columns.
\end{remark}

\subsection{Construction from a class of systematic MRD codes}

To construct systematic MRD codes, we need the following theorem, which provides a criterion for linear MRD codes.

\begin{theorem}{\rm \cite{tm}}\label{thm:cri-MRD}
Let $m\geq n$. Let $\textbf{G} \in \mathbb F_{q^m}^{k \times n}$ be a generator matrix of a linear rank-metric code $\mathcal C \subseteq \mathbb F_{q^m}^n$. Then $\mathcal C$ is an MRD code if and only if for any $\textbf{B} \in UT^*_n(q)$ every maximal minor of $\textbf{GB}$ is nonzero, where $UT^*_n(q)$ denotes the set of all $n\times n$ upper triangular matrices over $\mathbb F_{q}$ whose main diagonal elements are all $1$.
\end{theorem}

\begin{lemma}\label{lem:sys MRD}
Let $q$ be a prime power. Let $m,n$ and $\delta$ be positive integers satisfying $m>n\geq \delta\geq 2$. Let $k=n-\delta+1$ and $m\geq kn-k^2+2$. Let $(1, \beta, \beta^2, \ldots, \beta^{m-1} )$ be an ordered polynomial basis of $\mathbb F_{q^m}$ over $\mathbb F_q.$  If there exists a $k \times n$ matrix
\begin{center}
$\textbf{G}$=$\left(
               \begin{array}{cccccccccc}
                 1 &  &  &  &  & a_{1, k}\beta^k & a_{1, k+1}\beta^{k+1} & \cdots & a_{1, n-2}\beta^{n-2} & a_{1, n-1}\beta^{n} \\
                  & 1 &  &  &  & a_{2, k}\beta^{k-1} & a_{2, k+1}\beta^k & \cdots & a_{2, n-2}\beta^{n-3} & a_{2, n-1}\beta^{n-2} \\
                  &  & \ddots &  &  & \vdots & \vdots & \ddots & \vdots & \vdots \\
                  &  &  & 1 &  & a_{k-1, k}\beta^2 & a_{k-1, k+1}\beta^3 & \cdots & a_{k-1,n-2}\beta^{n-k} & a_{k-1, n-1}\beta^{n-k+1} \\
                  &  &  &  & 1 & a_{k, k}\beta & a_{k, k+1}\beta^2 & \cdots & a_{k, n-2}\beta^{n-k-1} & a_{k, n-1}\beta^{n-k} \\
               \end{array}
             \right) 
 $,
\end{center}
where $a_{i, j} \in \mathbb F^*_q$, $i \in \{1, 2, \ldots, k\}$ and $j \in \{k, k+1, \ldots, n-1\}$, such that every minor of the matrices
\begin{center}
$\textbf{A}_1=\left(
  \begin{array}{ccc}
    a_{1,k} & \cdots & a_{1,n-2} \\
    \vdots & \ddots & \vdots \\
    a_{k,k} & \cdots & a_{k,n-2} \\
  \end{array}
\right)$ and
$\textbf{A}_2=\left(
  \begin{array}{ccc}
    a_{2,k} & \cdots & a_{2,n-1} \\
    \vdots & \ddots & \vdots \\
    a_{k,k} & \cdots & a_{k,n-1} \\
  \end{array}
\right)$
\end{center}
is nonzero, then $\textbf{G}$ is a generator matrix of a systematic MRD$[m\times n,\delta]_q$ code.
\end{lemma}

\begin{proof}
Obviously, $n-k\geq 0$. By Theorem \ref{thm:cri-MRD}, it suffices to prove that for any $\textbf{B}$$\in$$UT^*_n(q)$, every $k$-minor of $\textbf{GB}$ is nonzero. To ensure smooth reading of the paper, we move the proof to Appendix \ref{app-A}. \qed
\end{proof}

For a vector $(v_1,v_2,\ldots,v_n)$ of length $n$, if its rightmost nonzero component is $v_r$ for some $1\leq r\leq n$, then $r$ is said to be the {\em valid length} of this vector.

\begin{construction}\label{con:from sys MRD}
Let $m,n$ and $\delta$ be positive integers satisfying $m>n\geq \delta\geq 2$. Let $k=n-\delta+1$ and $m\geq kn-k^2+2.$ If there exists a $k\times n$ matrix $\textbf{G}$ satisfying the condition in Lemma $\ref{lem:sys MRD}$ such that $\textbf{G}$ is a generator matrix of a systematic MRD$[m\times n,\delta]_q$ code, then there exists an optimal $[\mathcal F, \sum_{i=0}^{k-1} \gamma_i, \delta]_q$ code $\mathcal C$ for any $m'\times n$ Ferrers diagram $\mathcal F$ satisfying
\begin{itemize}
\item[$(1)$] $\gamma_{i}=\min \{\max\{\gamma_l+i-l: l\in [k]\},m\}$ for any $k\leq i \leq n-2$,
\item[$(2)$] $m'=\min\{\max\{\gamma_0+n,\max\{\gamma_l+n-1-l:1\leq l\leq k-1\}\},m\}$,
\end{itemize}
where $\gamma_i$, $i\in [n]$, is the number of dots in the $i$-th column of $\cal F$.
\end{construction}

\begin{proof} Start from the generator matrix $\textbf{G}$ of the given systematic MRD code. We can apply Lemma \ref{lem:subcode from MRD} by setting $\lambda_i=\gamma_i$, $i\in [k]$, to obtain an optimal FDRM code $\mathcal C$ in some Ferrers diagram $\mathcal F$ with dimension $\sum_{i=0}^{k-1} \gamma_i$ and rank at least $\delta$. It suffices to analyze the number of dots in each column of $\mathcal F$.

By Lemma \ref{lem:subcode from MRD}, for any $\textbf{c}=(c_0,c_1,\ldots,c_{n-1})\in \mathcal C$, we have $\textbf{c}=\textbf{u}\textbf{G}$ for some $\textbf{u}=(u_0,u_1,\ldots,u_{k-1})$.

When $0\leq i\leq k-1$, $c_i=u_i,$ and so $\Psi_m(c_i)=\Psi_m(u_i)=(u_{i,0},\ldots,u_{i,\gamma_i-1},0,\ldots,0)$. Thus, the $i$-th column of $\mathcal F$ has $\gamma_i$ dots.

When $k\leq i\leq n-2$, $c_i=\sum_{l=0}^{k-1} u_l a_{l+1,i}\beta^{i-l}$ and so $\Psi_m(c_i)=\sum_{l=0}^{k-1} a_{l+1,i}\Psi_m(u_l\beta^{i-l})$. For $l\in [k]$, $\Psi_{m}(u_l)=(u_{l,0},u_{l,1},\ldots,u_{l,\gamma_{l}-1},0,\ldots,0)^T$ implies  $u_l=u_{l,0}+u_{l,1}\beta+\cdots+u_{l,\gamma_l-1}\beta^{\gamma_l-1}$. Note that $\beta^{m+j}$ can be written as a linear combination of $1,\beta,\beta^2,\cdots,\beta^{m-1}$ for any nonnegative integer $j$. It follows that for each $l\in [k]$, as a vector of length $m$, $\Psi_m(u_l\beta^{i-l})$ has a valid length of at most $\min\{\gamma_l+i-l,m\}$. Thus $\Psi_m(c_i)$ has a valid length of at most $\max\{\gamma_l+i-l:l\in[k]\}$ if $\max\{\gamma_l+i-l:l\in[k]\}\leq m$, or $m$ otherwise, which coincides with Condition $(1)$.

When $i=n-1$, $c_{n-1}=\sum_{l=1}^{k-1} u_l a_{l+1,n-1}\beta^{n-1-l}+u_{0}a_{1,n-1}\beta^n$ and so $\Psi_m(c_{n-1})=\sum_{l=1}^{k-1} a_{l+1,n-1}\Psi_m(u_l\beta^{n-1-l})+a_{1,n-1}\Psi_m(u_{0}\beta^n)$. As a vector of length $m$, $\Psi_m(u_0\beta^{n})$ has a valid length of at most $\min\{\gamma_0+n,m\}$. For each $1\leq l\leq k-1$, $\Psi_m(u_l\beta^{n-1-l})$ has a valid length of at most $\min\{\gamma_l+n-1-l,m\}$. Thus $\Psi_m(c_i)$ has a valid length of at most $\max \{\gamma_0+n,\gamma_1+n-1-1,\gamma_2+n-1-2,\ldots,\gamma_{k-1}+n-1-(k-1)\}$ if $\max \{\gamma_0+n,\gamma_1+n-1-1,\gamma_2+n-1-2,\ldots,\gamma_{k-1}+n-1-(k-1)\}\leq m$, or $m$ otherwise, which coincides with Condition $(2)$. \qed
\end{proof}

\begin{example}\label{eg:sys MDS-1}
Let $q$ be a prime power. Let $m\geq 2n-2$ and $3\leq n\leq q+2$. Let $(1, \beta, \ldots, \beta^{m-1} )$ be an ordered polynomial basis of $\mathbb F_{q^m}$ over $\mathbb F_q.$  Construct a $2 \times n$ matrix
\begin{center}
 $\textbf{G}=\left(
               \begin{array}{cccccccc}
                 1 &   & \beta^2 & \beta^3 & \beta^4 & \cdots & \beta^{(n-2)} & \beta^n \\
                   & 1 & a_1\beta & a_2\beta^2 & a_3\beta^3 & \cdots & a_{n-3}\beta^{(n-3)} & \beta^{(n-2)} \\
               \end{array}
             \right)
 $,
\end{center}
where $a_{i} \in \mathbb F_q^*$, $1\leq i\leq n-3$, and $a_i\neq a_j$ for any $i\neq j$. By Lemma $\ref{lem:sys MRD}$, $\textbf{G}$ is a generator matrix of a systematic MRD$[m \times n, n-1]_{q}$ code.

Let $\cal F$ be an $m'\times n$ Ferrers diagram satisfying $\gamma_{0}\leq m-n$, $\gamma_i=\gamma_0+i$ for $1\leq i\leq n-2$ and $\gamma_{n-1}=m'=\gamma_{0}+n$. It is readily checked that $\cal F$ satisfies Conditions $(1)$ and $(2)$ in Construction $\ref{con:from sys MRD}$. Thus there exists an optimal $[\mathcal F, \gamma_0+\gamma_1, n-1]_q$ code.
\end{example}

\begin{remark}
When $q=n-2$, Example $\ref{eg:sys MDS-1}$ cannot be obtained from Theorem $\ref{thm:from MDS}$ since no $[n, 2,n-1]_{n-2}$ MDS code exists, even though here Theorem $\ref{thm:from MDS}$ can deal with all cases of prime power $q\geq n-1$. No known construction can be applied here to handle the case of $q=n-2$. It is readily checked that Theorems $\ref{thm:shortening}$-$\ref{thm:from Max}$ are invalid. Consider Theorem $\ref{the:from sys MRD-new1}$ with $l\geq 2$ $($when $l=1$, Theorem $\ref{the:from sys MRD-new1}$ degenerates into Theorem $\ref{thm:shortening})$. Condition $(1)$ yields $t_1\geq \gamma_1=\gamma_0+1$. Condition $(3)$ yields $\gamma_{t_1}\geq t_2\geq 2t_1$ $($note that $t_1\mid t_2$ and $t_1<t_2)$. Since $\gamma_{t_1}=\gamma_0+t_1$, we have $t_1\leq \gamma_0$, a contradiction. Thus Theorem $\ref{the:from sys MRD-new1}$ is also invalid. When $\gamma_0>1$, Example $\ref{eg:sys MDS-1}$ cannot be obtained from Theorem $\ref{thm:upper triangular}$. By using Lemma $\ref{lem:dis combin}$ $($note that $\max_{n-\delta+1\leq i\leq n-1}(\gamma_i-\gamma_{i-1})=2$ and $v_0=\gamma_0+\gamma_1=2\gamma_0+1)$, we have that Theorem $\ref{thm:combine with same dim}$ is invalid. Similar arguments hold for the following two examples.
\end{remark}

\begin{example}\label{eg:sys MDS-2}
Let $q=5$, $n=7$ and $m\geq 14$. Let $(1, \beta, \ldots, \beta^{m-1} )$ be an ordered polynomial basis of $\mathbb F_{5^m}$ over $\mathbb F_5.$  Construct a $3 \times 7$ matrix
\begin{center}
$\textbf{G}=\left(
              \begin{array}{ccccccc}
                1 &   &   & 2\beta^3 & 3\beta^4 & 4\beta^5 & \beta^7 \\
                  & 1 &   & \beta^2 & \beta^3 & \beta^4 & \beta^5 \\
                  &   & 1 & \beta & 2\beta^2 & 4\beta^3 & 3\beta^4 \\
              \end{array}
            \right)$.
\end{center}
By Lemma $\ref{lem:sys MRD}$, $\textbf{G}$ is a generator matrix of a systematic MRD$[m \times 7, 5]_5$ code.

Let $\cal F$ be an $m'\times n$ Ferrers diagram satisfying $\gamma_{0}\leq m-7$, $\gamma_i=\gamma_0+i$ for $1\leq i\leq 5$ and $\gamma_{6}=m'=\gamma_{0}+7$. It is readily checked that $\cal F$ satisfies Conditions $(1)$ and $(2)$ in Construction $\ref{con:from sys MRD}$. Thus there exists an optimal $[\mathcal F, \gamma_0+\gamma_1+\gamma_2, 5]_5$ code.
\end{example}

\begin{example}\label{eg:sys MDS-3}
Let $q=7$, $n=9$ and $m\geq 20$. Let $(1, \beta, \ldots, \beta^{m-1} )$ be an ordered polynomial basis of $\mathbb F_{7^m}$ over $\mathbb F_7.$  Construct a $3 \times 9$ matrix
\begin{center}
$\textbf{G}=\left(
              \begin{array}{ccccccccc}
                1 &   &   & 2\beta^3 & 3\beta^4 & 4\beta^5 & 5\beta^6 & 6\beta^7 & \beta^9 \\
                  & 1 &   & \beta^2 & \beta^3 & \beta^4 & \beta^5 & \beta^6 & \beta^7\\
                  &   & 1 & \beta & 6\beta^2 & 3\beta^3 & 5\beta^4 & 2\beta^5 & 4\beta^6\\
              \end{array}
            \right).$
\end{center}
By Lemma $\ref{lem:sys MRD}$, $\textbf{G}$ is a generator matrix of a systematic MRD$[m \times 9, 7]_7$ code.

Let $\cal F$ be an $m'\times n$ Ferrers diagram satisfying $\gamma_{0}\leq m-9$, $\gamma_i=\gamma_0+i$ for $1\leq i\leq 7$ and $\gamma_{8}=m'=\gamma_{0}+9$. It is readily checked that $\cal F$ satisfies Conditions $(1)$ and $(2)$ in Construction $\ref{con:from sys MRD}$. Thus there exists an optimal $[\mathcal F, \gamma_0+\gamma_1+\gamma_2, 7]_7$ code.
\end{example}

\begin{remark}
In Lemma $\ref{lem:sys MRD}$, the top right entry of $\textbf{G}$ uses the $n$-th power of $\beta$, which deviates from the patter in the rest of $\textbf{G}$. If $a_{1,n-1}\beta^{n-1}$ is taken as the top right entry in $\textbf{G}$, then similar arguments to those in the proof of Lemma $\ref{lem:sys MRD}$ show that $\textbf{G}$ is a generator matrix of a systematic MRD$[m\times n,\delta]_q$ code if every minor of the matrix
\begin{center}
$\left(
\begin{array}{ccccc}
a_{1,k} & a_{1,k+1} & \cdots & a_{1,n-2} & a_{1,n-1} \\
a_{2,k} & a_{2,k+1} & \cdots & a_{2,n-2} & a_{2,n-1} \\
\vdots & \vdots & \ddots & \vdots & \vdots \\
a_{k,k} & a_{k,k+1} & \cdots & a_{k,n-2} & a_{k,n-1} \\
\end{array}
\right)$
\end{center}
is nonzreo. However, so far we have not found any appropriate $\textbf{G}$ such that new optimal FDRM codes can be derived from it.
\end{remark}

\subsection{Construction based on subcodes of Gabidulin codes}

For any positive integer $i$ and any $a\in \mathbb F_{q^m}$, set $a^{[i]} \triangleq a^{q^i}$. In this section, we shall generalize Construction 2 in \cite{egrw} by exploring subcodes of Gabidulin codes.

Let $m\geq n$ and $q$ be any prime power. A Gabidulin code $\mathcal{G}[m\times n,\delta]_q$ is an MRD$[m\times n,\delta]_q$ code whose generator matrix $\textbf{G}$ in vector representation is
\begin{center}
$\textbf{G}=\left(
               \begin{array}{cccc}
                 g_0 & {g_1} & \cdots & g_{n-1} \\
                 g_0^{[1]} & g_1^{[1]} & \cdots & g_{n-1}^{[1]} \\
                 \vdots & \vdots & \ddots & \vdots \\
                 g_0^{[n-\delta]} & g_1^{[n-\delta]} & \cdots & g_{n-1}^{[n-\delta]} \\
               \end{array}
             \right)
$,
\end{center}
where $g_0, g_1, \ldots, g_{n-1} \in \mathbb F_{q^m}$ are linearly independent over $\mathbb F_q$ (see \cite{g}).

The following lemma is a generalization of Lemma 5 in \cite{egrw}, which only deals with the case of $r=1$. We move its proof to Appendix \ref{app-B}.

\begin{lemma}\label{lem:subcode from Gab}
Let $\eta,r,d,\kappa$ and $\mu$ be positive integers such that $\kappa=\eta-r-d+1$, $r<\kappa$ and $\eta\leq \mu+r$. Then there exists a matrix $ \textbf{G}\in \mathbb F_{q^{\mu}}^{\kappa \times \eta}$ of the following form
\begin{center}\begin{scriptsize}
$\left(
                           \begin{array}{cccccccccccccc}
                              1 &   &   &   &   &   &   & \alpha_{0,\kappa} & \cdots & \alpha_{0,\eta-r-1} & 0 & 0 & \cdots & 0 \\
                                & 1  &   &   &   &   &   & \alpha_{1,\kappa} & \cdots & \alpha_{1,\eta-r-1} & \alpha_{1,\eta-r} & 0 & \cdots & 0 \\
                                &   & \ddots  &   &  &    &   & \vdots & \ddots & \vdots & \vdots & \vdots & \ddots & \vdots \\
                                &   &   & 1  &   &   &   & \alpha_{r-1,\kappa} & \cdots & \alpha_{r-1,\eta-r-1} & \alpha_{r-1,\eta-r} & \alpha_{r-1,\eta-r+1} & \cdots & 0\\
                                &   &   &   & 1  &   &   & \alpha_{r,\kappa} & \cdots & \alpha_{r,\eta-r-1} & \alpha_{r,\eta-r} & \alpha_{r,\eta-r+1} & \cdots & \alpha_{r,\eta-1}\\
                                &   &   &   &   & \ddots  &   & \vdots & \ddots & \vdots & \vdots & \vdots & \ddots & \vdots\\
                                &   &   &   &   &   & 1 & \alpha_{\kappa-1,\kappa} & \cdots & \alpha_{\kappa-1,\eta-r-1}  & \alpha_{\kappa-1,\eta-r} & \alpha_{\kappa-1,\eta-r+1} & \cdots & \alpha_{\kappa-1,\eta-1} \\
                           \end{array}
                         \right)$\end{scriptsize}
\end{center}
\noindent satisfying that for each $0\leq i\leq r$, the sub-matrix obtained by removing the first $i$ rows, the leftmost $i$ columns and the rightmost $r-i$ columns of $\textbf{G}$ can produce an MRD$[\mu \times (\eta-r),d+i]_q$ code.
\end{lemma}

\begin{construction}\label{con:subcode from Gab}
Let $\delta$, $n$ and $r$ be positive integers satisfying $r+1\leq \delta \leq n-r$. To take a $\kappa \times \eta$ matrix $\textbf{G}$ satisfying Lemma $\ref{lem:subcode from Gab}$, assume that   $d=\delta-r$, $\kappa=n-\delta+1$, $\eta=n$ and $\mu=n-r$. Let $\mathcal{F}$ be an $m\times n$ Ferrers diagram, each of whose rightmost $\delta-1$ columns has at least $n-r$ dots. Let $\gamma_i$, $i\in [n]$, is the number of dots in the $i$-th column of $\mathcal F$. For $l\in [r]$, set $s_l= \min \{ \gamma_l-1, \gamma_{n-r+l}-n+r-\sum_{j=0}^{l-1}(s_j+1)-1 \}$. Let
\[
\begin{split}
\textbf{U}=& \left\{(u_0,\ldots,u_{\kappa-1}) \in \mathbb F_{q^{n-r}}^{\kappa}:\Psi_{n-r}(u_l)=(u_{l,0},\ldots,u_{l,s_l},0,\ldots,0)^T {~\rm for~} l\in [r], \right.\\
& \left. \Psi_{n-r}(u_l)=(u_{l,0},\ldots,u_{l,\gamma_{l}-1},0,\ldots,0)^T {~\rm for~} r\leq l\leq \kappa-1, {\rm ~all~ possible~} u_{i,j} \in \mathbb F_q \right\}.
\end{split}
\]
Let $\overline{\Psi}_{n-r}(u_l)=(u_{l,0},\ldots,u_{l,s_l})^T$ for $l\in [r]$. If $0\leq s_0\leq s_1 \leq \cdots \leq s_{r-1}$, then ${\cal C}=$
\begin{center}\small
$\left\{
\left(
\begin{array}{c}
\displaystyle\frac{
\begin{array}{c}
\Psi_{n-r}( \textbf{c})
\end{array}
      }{
\begin{array}{ccccccc}
0 & \cdots & 0 & \overline{\Psi}_{n-r}(u_0)  & \overline{\Psi}_{n-r}(u_1)  &  \cdots & \overline{\Psi}_{n-r}(u_{r-1}) \\
0 & \cdots  & 0 & 0 & \overline{\Psi}_{n-r}(u_0)  & \ddots & \overline{\Psi}_{n-r}(u_{r-2}) \\
\vdots &   & \vdots & \vdots &  \vdots & \ddots & \vdots \\
0 & \cdots & 0 & 0 &  0 &  \cdots & \overline{\Psi}_{n-r}(u_0) \\
\end{array} } \\
    \end{array}
  \right)\in \mathbb F^{m\times n}_q: \textbf{c}=\textbf{u}\cdot\textbf{G}, \textbf{u} \in \textbf{U}
\right\}$
\end{center}
is an $[\mathcal F, k, \delta]_q$ code for any prime power $q$, where $k=\sum_{i=0}^{r-1} s_i+r+ \sum_{i=r}^{n-\delta} \gamma_i$. Note that $m=n+\sum_{i=0}^{r-1} s_i$.
\end{construction}

\begin{proof}
One can easily verify the linearity and the dimension of the code. It remains to examine the minimum rank weight of any nonzero codeword $\textbf{C}$ from $\mathcal C$. Note that $\delta\leq n-r$, so $r\leq n-\delta<n-\delta+1=\kappa$.

Let $\textbf{C}$ be formed by $\textbf{c}=\textbf{uG}=(u_0, u_1, \ldots, u_{\kappa-1}) \textbf{G}$. Let $i^*=\min\{i:i\in[\kappa],u_i\neq 0, u_j=0 {~\rm for~any~} j<i\}$. Then $\textbf{c}=(0, \ldots,0,u_{i^*}, \ldots, u_{\kappa-1}) \textbf{G}$.

(1) If $i^*< r$, then let $\Psi^*_{n-r}(\textbf{uG})$ be an $(n-r)\times (n-r)$ matrix obtained by removing the leftmost $i^*$ columns and the rightmost $r-i^*$ columns of $\Psi_{n-r}(\textbf{uG})$. By Lemma \ref{lem:subcode from Gab}, $\Psi_{n-r}^*(\textbf{uG})$ is a codeword of an MRD$[\mu \times (n-r),\delta-r+i^*]_q$ code, whose generator matrix can be obtained by removing the first $i^*$ rows, the leftmost $i^*$ columns and the rightmost $r-i^*$ columns of $\textbf{G}$. Thus rank$(\Psi_{n-r}^*(\textbf{uG}))\geq \delta-r+i^*$.

Furthermore, under the broken line of $\textbf{C}$, since $\overline{\Psi}_{n-r}(u_{i^*})$ with  the length $s_{i^*}$ is a nonzero vector, the rightmost $r-i^*$ columns contribute rank $r-i^*$. Therefore, rank$(\textbf{C})\geq {\rm rank}(\Psi_{n-r}^*(\textbf{uG}))+r-i^* \geq \delta-r+i^*+r-i^*=\delta$.

$(2)$ If $i^*\geq r$, then let $\Psi^*_{n-r}(\textbf{uG})$ be an $(n-r)\times (n-r)$ matrix obtained by removing the leftmost $r$ columns of $\Psi_{n-r}(\textbf{uG})$. By Lemma \ref{lem:subcode from Gab}, $\Psi_{n-r}^*(\textbf{uG})$ is a codeword of an MRD$[\mu \times (n-r),\delta]_q$ code, whose generator matrix can be obtained by removing the first $r$ rows and the leftmost $r$ columns of $\textbf{G}$. Thus rank$(\textbf{C})\geq\delta$. \qed
\end{proof}

\begin{theorem}\label{thm:subcodes from Gab}
Let $\delta$, $n$ and $r$ be positive integers satisfying $r+1\leq \delta \leq n-r$. Let $\mathcal{F}$ be an $m\times n$ Ferrers diagram satisfying that
\begin{itemize}
\item[$(1)$] $\gamma_{n-\delta}\leq n-r$,
\item[$(2)$] $\gamma_{n-\delta+1}\geq n-r$,
\item[$(3)$] $\gamma_{n-r+l}\geq n-r+\sum_{j=0}^{l} \gamma_j$ for $l\in [r]$.
\end{itemize}
Then there exists an optimal $[\mathcal F, \sum_{i=0}^{n-\delta} \gamma_i, \delta]_q$ code for any prime power $q$.
\end{theorem}

\begin{proof} Note that $r+1\leq \delta$ yields $n-\delta\leq n-r-1$. Consider a new Ferrers diagram $\mathcal F'$ with $\gamma'_i$ for $i\in[n]$ as the number of dots in its $i$-th column, satisfying
\begin{equation}
\label{}\nonumber
\gamma'_i=\left \{
\begin {aligned}
&\gamma_i,~~{\rm if}~0\leq i\leq n-\delta;\\
&n-r,~~{\rm if}~n-\delta+1\leq i\leq n-r-1;\\
&n-r+\sum_{j=0}^{i+r-n} \gamma_j,~~{\rm if}~n-r\leq i\leq n-1.
\end {aligned}
\right.
\end{equation}
Then $\mathcal F'$ is a sub-Ferrers diagram of $\mathcal F$ and $\gamma'_{i}\leq n-r$ for $0\leq i\leq n-r-1$. According to Lemma \ref{lem:upper bound}, it suffices to show that there exists an optimal $[\mathcal F',\sum_{i=0}^{n-\delta} \gamma'_i, \delta]_q$ code, which implies the existence of an optimal $[\mathcal F, \sum_{i=0}^{n-\delta} \gamma_i, \delta]_q$ code.

Clearly, each of the rightmost $\delta-1$ columns of $\mathcal F'$ has at least $n-r$ dots. To apply Construction \ref{con:subcode from Gab}, we need to count $s_l$ for $l\in[r]$. By Condition $(3)$, $s_0= \min \{ \gamma'_0-1, \gamma'_{n-r}-n+r-1 \}=\gamma'_0-1$. It follows that by induction on $l$, $l\in[r]$, we have  $s_l= \min \{ \gamma'_l-1, \gamma'_{n-r+l}-n+r-\sum_{j=0}^{l-1}(s_j+1)-1 \}=\min \{ \gamma'_l-1, \gamma'_{n-r+l}-n+r-\sum_{j=0}^{l-1}\gamma'_j-1 \}=\gamma'_l-1$. Thus $0\leq s_0\leq s_1 \leq \cdots \leq s_{r-1}$, and Construction \ref{con:subcode from Gab} provides an $[\mathcal F', \sum_{i=0}^{n-\delta} \gamma'_i, \delta]_q$ code, which is optimal by Lemma \ref{lem:upper bound}. \qed
 \end{proof}

We remark that as a corollary of Theorem \ref{thm:subcodes from Gab} with $r=1$, we can obtain Theorem 8 in \cite{egrw}.

\begin{example}\label{eg:subcodes from Gab}
Let $n\geq 3$ and
\begin{center}
 $\mathcal F=\begin{array}{ccccccccc}
               \bullet & \bullet & \bullet & \bullet & \cdots & \bullet &\bullet & \bullet & \bullet \\
               \bullet & \bullet & \bullet & \bullet & \cdots & \bullet &\bullet &\bullet & \bullet \\
                &  & \bullet & \bullet & \cdots & \bullet &\bullet & \bullet & \bullet \\
                &  & \bullet & \bullet & \cdots & \bullet &\bullet & \bullet & \bullet \\
                &  &  &  & \ddots & \vdots & \vdots & \vdots & \vdots \\
                &  &  &  &  & \bullet &\bullet & \bullet & \bullet \\
                &  &  &  &  & \bullet &\bullet & \bullet & \bullet \\
                &  &  &  &  &  &  &\bullet & \bullet \\
                &  &  &  &  &  &  &\bullet & \bullet \\
                &  &  &  &  &  &  & & \bullet \\
                &  &  &  &  &  &  & & \bullet
             \end{array}
 $
\end{center}
be a $(2n+2) \times 2n$ Ferrers diagram. Take $\delta=4$. Apply Theorem $\ref{thm:subcodes from Gab}$ with $r=2$. One can check that the rightmost $3$ columns of $\mathcal F$ have at least $2n-2$ dots, $\gamma_{2n-2}=2n=2n-2+\gamma_0$ and $\gamma_{2n-1}=2n+2=2n-2+\gamma_0+\gamma_1$.
So $\mathcal F$ satisfies the conditions of Theorem $\ref{thm:subcodes from Gab}$, and an optimal $[\mathcal F,2(n-1)^2,4]_q$ code exists for any prime power $q$.
\end{example}

\begin{remark}
No known construction can be applied to obtain Example $\ref{eg:subcodes from Gab}$. It is readily checked  that Theorems $\ref{thm:shortening}$-$\ref{thm:from Max}$ and $\ref{thm:upper triangular}$ are invalid. Consider Theorem $\ref{the:from sys MRD-new1}$ with $l\geq 2$. Condition $(1)$ yields $t_1\geq \gamma_{2n-4}=2n-2$ and Condition $(3)$ yields $\gamma_{t_1}\geq t_2\geq 2t_1\geq 4n-4$. Due to $\gamma_{t_1}\leq 2n+2$, we have $n\leq 3$. So Theorem $\ref{the:from sys MRD-new1}$ is invalid when $n>3$. Actually by more careful calculation, one can check that Theorem $\ref{the:from sys MRD-new1}$ is also invalid for $n=3$. We leave the details to the interested reader. Theorem $\ref{thm:from MDS}$ can provide an $[\mathcal F,2n^2-4n,4]_q$ code for any prime power $q\geq 2n-1$, but it is not optimal. By using Lemma $\ref{lem:dis combin}$ $($note that $\max_{n-\delta+1\leq i\leq n-1}(\gamma_i-\gamma_{i-1})=2$ and $v_0=2(n-1)^2)$, one can see that Theorem $\ref{thm:combine with same dim}$ is invalid.
\end{remark}

\section{New Ferrers diagram rank-metric codes from old}

First we give a slight variation of Theorem \ref{thm:combine with same dim}. It is not complicated, but inspires us to establish this section.

\begin{theorem}\label{thm:combine generalized}
Let $\mathcal F_i$ for $i=1,2$ be an $m_i \times n_i$ Ferrers diagram, and $\mathcal C_i$ be an $[\mathcal F_i, k_i, \delta_i]_q$ code. Let $\mathcal D$ be an $m_3 \times n_3$ Ferrers diagram and $\mathcal C_3$ be a $[\mathcal D, k_3, \delta]_q$ code, where $m_3 \geq m_1$ and $n_3 \geq n_2$. Let $m=m_2+m_3$ and $n=n_1+n_3$. Let
\begin{center}
$\mathcal F=\left(
  \begin{array}{cc}
    \mathcal F_1 & \hat{\mathcal D} \\
      & \mathcal F_2 \\
  \end{array}
\right)$
\end{center}
be an $m \times n$ Ferrers diagram $\mathcal F$, where $\hat{\mathcal{D}}$ is obtained by adding the fewest number of new dots to the lower-left corner of ${\cal D}$ such that ${\cal F}$ is a Ferrers diagram. Then there exists an $[\mathcal F, \min\{k_1,k_2\}+k_3, \min\{\delta_1+\delta_2, \delta\}]_q$ code $\cal C$ satisfying that for any codeword $\textbf{C}\in {\cal C}$, $\textbf{C}|_{\mathcal F_1}=\boldsymbol{O}$ if and only if $\textbf{C}|_{\mathcal F_2}=\boldsymbol{O}$, where $\textbf{C}|_{\mathcal F_i}$ denotes the restriction of $\textbf{C}$ in ${\mathcal F_i}$ for $i=1,2$.
\end{theorem}

\begin{proof}
Obviously, ${\cal C}_3$ is also a $[\hat{\mathcal{D}}, k_2, \delta]_q$ code, and the existence of $[\mathcal F_i, k_i, \delta_i]_q$ code $\mathcal C_i$ for $i=1,2$ implies the existence of $[\mathcal F_i, \min\{k_1,k_2\}, \delta_i]_q$ code $\mathcal C'_i$. Let $\varphi$ be an isomorphism $\varphi : \mathcal C'_1\longrightarrow \mathcal C'_2$ (in the sense of linear spaces) and set
\begin{center}
  $\mathcal C=\begin{Bmatrix}
  \left(
    \begin{array}{cc}
      \textbf{X} & \textbf{D} \\
      \textbf{0} & \varphi(\textbf{X}) \\
    \end{array}
  \right): \textbf{X} \in \mathcal C'_1, \textbf{D} \in \mathcal C_3
  \end{Bmatrix}$.
\end{center}
Clearly $\mathcal C$ is a linear code of dimension $\min\{k_1,k_2\}+k_3$. It suffices to examine the minimum rank distance of $\cal C$. Take any nonzero codeword $\textbf{C}$ from $\mathcal C$. Since $\textbf{X}$ and $\varphi(\textbf{X})$ are either both zero or both nonzero, we consider the following two cases:

$\bullet$ $\textbf{X}\neq \textbf{0}$.
\begin{center}
  $rank(\textbf{C}) = rank
  \left(
    \begin{array}{cc}
      \textbf{X} & \textbf{D} \\
      \textbf{0} & \varphi(\textbf{X}) \\
    \end{array}
  \right)\geq  rank(\textbf{X})+rank(\varphi(\textbf{X})) \geq \delta_1+\delta_2$.
\end{center}

$\bullet$ $\textbf{X}= \textbf{0}$.
\begin{center}
 $rank(\textbf{C})=rank(\textbf{D}) \geq \delta$.
\end{center}
Therefore, $\mathcal C$ is an $[\mathcal F, \min\{k_1,k_2\}+k_3, \min\{\delta_1+\delta_2, \delta\}]_q$ code.
\end{proof} \qed

\subsection{Generalization of Theorem \ref{thm:combine generalized}}

To obtain optimal FDRM codes, in the process of using Theorem \ref{thm:combine generalized}, it is often required that $\mathcal C_3$ is an optimal $[\mathcal D, k_2, \delta]_q$ code. If the optimality of $\mathcal C_3$ is unknown, then what shall we do? A natural idea is to remove a sub-diagram from $\mathcal D$ to obtain a new Ferrers diagram $\mathcal D'$ such that the FDRM code in $\mathcal D'$ is optimal, and then mix the removed sub-diagram to ${\cal F}_1$ or ${\cal F}_2$. We shall illustrate the idea by using Example \ref{eg:com-1}. Note that for this example, one can check that any known constructions cannot attain the required dimension.

First, we introduce a new concept. Let $\mathcal F_1$ be an $m_1\times n_1$ Ferrers diagram, $\mathcal F_2$ be an $m_2\times n_2$ Ferrers diagram and $\mathcal F$ be an $m\times n$ Ferrers diagram. Let $\phi_l$ for $l\in \{1,2\}$ be an injection from $\mathcal F_l$ to $\mathcal F$ (in the sense of set-theoretical language).
$\mathcal F$ is said to be a {\em proper combination} of $\mathcal F_1$ and $\mathcal F_2$ on a pair of mappings $\phi_1$ and $\phi_2$, if
\begin{itemize}
\item[(1)] $\phi_1(\mathcal F_1)\cap \phi_2(\mathcal F_2)=\varnothing;$
\item[(2)] $|\mathcal F_1|+|\mathcal F_2|=|\mathcal F|;$
\item[(3)] for any $l\in \{1,2\}$ and any two different elements $(i_{l,1},j_{l,1})$, $(i_{l,2},j_{l,2})$ of $\mathcal F_l$, set $\phi_l(i_{l,1},j_{l,1})=(i'_{l,1},j'_{l,1})$ and $\phi_l(i_{l,2},j_{l,2})=(i'_{l,2},j'_{l,2})$; $i'_{l,1}=i'_{l,2}$ or $j'_{l,1}=j'_{l,2}$ whenever $i_{l,1}=i_{l,2}$ or $j_{l,1}=j_{l,2}$.
\end{itemize}
Condition (3) means that if two dots in ${\cal F}_l$ for $l\in\{1,2\}$ are in the same row or same column,
then their corresponding two dots in $\cal F$ are also in the same row or same column.

\begin{example}\label{eg:}
Let
\begin{center}
$\mathcal F_1=\begin{array}{cc}
                     \bullet & \bullet \\
                      \bullet & \bullet \\
                       & \bullet \\
                       & \bullet \\
                       & \bullet \\
                       & \bullet
                  \end{array}
$ ~~~and~~~
$\mathcal F_2=\begin{array}{cc}
   {\color{green}{\bullet}} & {\color{green}{\bullet}} \\
    & {\color{green}{\bullet}} \\
    & {\color{green}{\bullet}}
 \end{array}
$
\end{center} be two Ferrers diagrams. Then all of
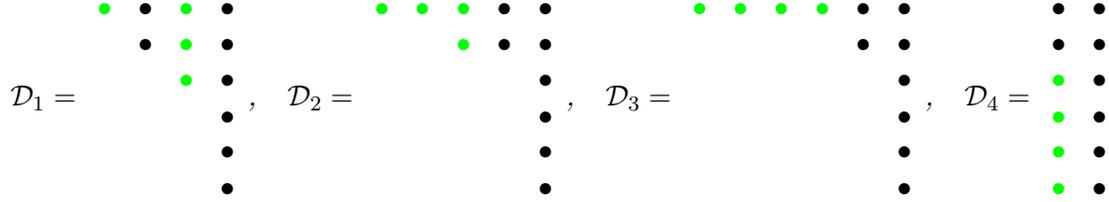
\begin{center}
$\mathcal D_1=\begin{array}{cccc}
                {\color{green}{\bullet}} & \bullet & {\color{green}{\bullet}} & \bullet \\
                 & \bullet & {\color{green}{\bullet}} & \bullet \\
                 &  & {\color{green}{\bullet}} & \bullet \\
                 &  &  & \bullet \\
                 &  &  & \bullet \\
                 &  &  & \bullet
              \end{array}
$, ~
$\mathcal D_2=\begin{array}{ccccc}
                {\color{green}{\bullet}} & {\color{green}{\bullet}} & {\color{green}{\bullet}} & \bullet & \bullet \\
                 &  & {\color{green}{\bullet}} & \bullet & \bullet \\
                 &  &  &  & \bullet \\
                 &  &  &  & \bullet \\
                 &  &  &  & \bullet \\
                 &  &  &  & \bullet
              \end{array}
$, ~
$\mathcal D_3=\begin{array}{cccccc}
                {\color{green}{\bullet}} & {\color{green}{\bullet}} & {\color{green}{\bullet}} & {\color{green}{\bullet}} & \bullet & \bullet \\
                 &  &  &  & \bullet & \bullet \\
                 &  &  &  &  & \bullet \\
                 &  &  &  &  & \bullet \\
                 &  &  &  &  & \bullet \\
                 &  &  &  &  & \bullet
              \end{array}
$, ~
$\mathcal D_4=\begin{array}{cc}
                \bullet  & \bullet  \\
                \bullet  & \bullet  \\
                {\color{green}{\bullet}} & \bullet  \\
                {\color{green}{\bullet}} & \bullet  \\
                {\color{green}{\bullet}} & \bullet  \\
                {\color{green}{\bullet}} & \bullet
              \end{array}
$
\end{center}
 are proper combinations of $\mathcal F_1$ and $\mathcal F_2$. Note that $\mathcal F_1$ keeps its shape invariant in $\mathcal D_l$ for any $l\in\{1,2,3,4\}$; $\mathcal F_2$ keeps its shape invariant in $\mathcal D_1$ and $\mathcal D_2$ $($the transpose of $\mathcal F_2$ is allowed$)$; $\mathcal F_2$ degenerates into a single row or column in $\mathcal D_3$ and $\mathcal D_4$.
\end{example}

\begin{proposition}\label{lem:prop}
Let $\mathcal F$ be a proper combination of Ferrers diagrams $\mathcal F_1$ and $\mathcal F_2$. Then for each $l\in\{1,2\}$, either $\mathcal F_l$ keeps its shape invariant in $\mathcal F$ $($the transpose of $\mathcal F_l$ is allowed$)$, or $\mathcal F_l$ degenerates into a single row or column in $\mathcal F$.
\end{proposition}

\proof For $l\in\{1,2\}$, if $\mathcal F_l$ only contains one row or one column, then the conclusion follows immediately. Assume that $R_1$ and $R_2$ are two different rows of $\mathcal F_l$, and $R_1$ contains at least two dots. It is readily checked that these two rows either keep their shape invariant in $\mathcal F$ (a transpose is allowed), or degenerate into a single row $R$ or column $C$ in $\mathcal F$. If it is the latter, then by considering the rightmost two dots of $R_1$ and $R_2$ in $\mathcal F_l$ we have that the rightmost column in $\mathcal F_l$ must degenerate into the row $R$ or column $C$ in $\mathcal F$ (we refer to it as Fact A), and for any dot $P$ in any row (if this row exists) of $\mathcal F_l$ other than $R_1$ and $R_2$, if $P$ is not in the rightmost column, then there exists one dot $P'$ in the rightmost column of $\mathcal F_l$ such that $P$ and $P'$ are in the same row, which yields that $P$ and $P'$ must degenerate into the same row or column in $\mathcal F$. Now it suffices to show that the row must be $R$ or the column must be $C$.

Since $R_1$ contains at least two dots, there exists a dot $P''$ in $\mathcal F_l$ such that (i) $P''$ is not in the rightmost column, (ii) $P''$ and $P$ ($P''$ could be $P$) are in the same row, and (iii) there exists a dot $P'''$ in $R_1$ such that $P''$ and $P'''$ are in the same column. Because of Fact A, $P'$ and $P'''$ degenerate into the row $R$ or the column $C$ in $\mathcal F$. This forces $P$, $P'$, $P''$ and $P'''$ to degenerate into the row $R$ or column $C$. Therefore, all dots in $\mathcal F_l$ must degenerate into the single row $R$ or column $C$ in $\mathcal F$. \qed

\begin{lemma}\label{lem:prop comb}
For $l\in\{1,2\}$, let $\mathcal F_l$ be an $m_l\times n_l$ Ferrers diagram and $\boldsymbol{M}_l$ be an $m_l\times n_l$ matrix whose entries not in $F_l$ are all zero. Let $\mathcal F$ be a proper combination of $\mathcal F_1$ and $\mathcal F_2$ on a pair of mappings $\phi_1$ and $\phi_2$ such that $\mathcal F$ is an $m\times n$ Ferrers diagram. Let $\boldsymbol{M}_{12}$ be an $m\times n$ matrix satisfying
\begin{equation}\label{}\nonumber
\boldsymbol{M}_{12}(i,j)=\left \{
\begin {aligned}
&\boldsymbol{M}_{1}(i_1,j_1),~~{\rm if}\ (i,j)=\phi_1(i_1,j_1) ~{\rm and}\ (i_1,j_1)\in {\mathcal F}_1;\\
&\boldsymbol{M}_{2}(i_2,j_2),~~{\rm if}\ (i,j)=\phi_2(i_2,j_2) ~{\rm and}\ (i_2,j_2)\in {\mathcal F}_2;\\
&0,~~{\rm otherwise}.
\end {aligned}
\right.
\end{equation}
Then rank$(\boldsymbol{M}_{12})\leq$ rank$(\boldsymbol{M}_{1})+$ rank$(\boldsymbol{M}_{2})$.
\end{lemma}

\proof For $l\in\{1,2\}$, denote by $\boldsymbol{M}_{12}|_{\mathcal F_l}$ the restriction of $\boldsymbol{M}_{12}$ in ${\mathcal F_l}$, i.e., $\boldsymbol{M}_{12}|_{\mathcal F_l}$ is an $m\times n$ matrix satisfying
\begin{equation}\label{}\nonumber
\boldsymbol{M}_{12}|_{\mathcal F_l}(i,j)=\left \{
\begin {aligned}
&\boldsymbol{M}_{12}(i,j),~~{\rm if}\ (i,j)\in \phi_l({\mathcal F}_l);\\
&0,~~{\rm otherwise}.
\end {aligned}
\right.
\end{equation}
Then using the basic fact that rank$(\boldsymbol{A}+\boldsymbol{B})\leq$ rank$(\boldsymbol{A})+$ rank$(\boldsymbol{B})$, we have rank$(\boldsymbol{M}_{12})\leq$ rank$(\boldsymbol{M}_{12}|_{\mathcal F_1})+$ rank$(\boldsymbol{M}_{12}|_{\mathcal F_2})$. By Proposition \ref{lem:prop}, rank$(\boldsymbol{M}_{12}|_{\mathcal F_l})\leq$ rank$(\boldsymbol{M}_{l})$ for $l\in\{1,2\}$. The conclusion is then straightforward. \qed

\begin{example}\label{eg:com-1}
We here construct an optimal $[\mathcal F,10,4]_q$ code $\mathcal C$ for any prime power $q$, where
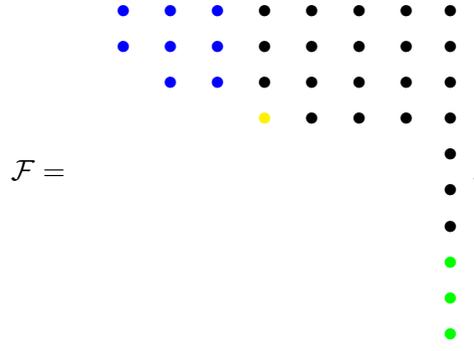
\begin{center}
$\mathcal{F}=
\begin{array}{ccccccccc}
 & {\color{blue}{\bullet}} & {\color{blue}{\bullet}} & {\color{blue}{\bullet}} & \bullet & \bullet & \bullet & \bullet & \bullet \\
 & {\color{blue}{\bullet}} & {\color{blue}{\bullet}} & {\color{blue}{\bullet}} & \bullet & \bullet & \bullet & \bullet & \bullet \\
 & ~~ & {\color{blue}{\bullet}} & {\color{blue}{\bullet}} & \bullet & \bullet & \bullet & \bullet & \bullet \\
 & ~~ & ~~& ~~ & {\color{yellow}{\bullet}} & \bullet & \bullet & \bullet & \bullet \\
   & ~ & ~~ & ~~ & ~~ & ~~ & ~~ & ~~ & \bullet \\
   & ~ & ~~ & ~~ & ~~ & ~~ & ~~ & ~~ & \bullet \\
 & ~ & ~~ & ~~ & ~~ & ~~ & ~~ & ~~ & \bullet \\
 & ~ & ~~ & ~~ & ~~ & ~~ & ~~ & ~~ & {\color{green}{\bullet}} \\
 & ~ & ~~ & ~~ & ~~ & ~~ & ~~ & ~~ & {\color{green}{\bullet}} \\
 & ~ & ~~ & ~~ & ~~ & ~~ & ~~ & ~~ & {\color{green}{\bullet}}
\end{array}$.
\end{center}
First take the following four Ferrers sub-diagrams of $\mathcal F$:
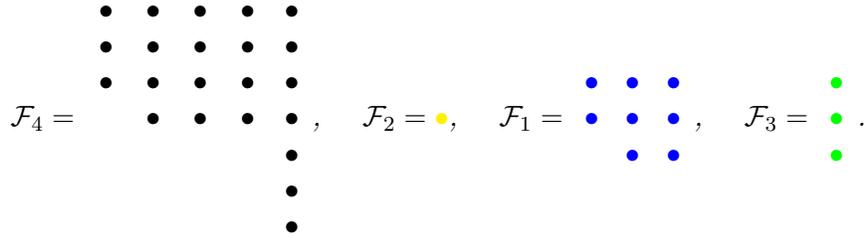
\begin{center}
$
\mathcal{F}_4=\begin{array}{ccccc}
                \bullet & \bullet & \bullet & \bullet & \bullet \\
                \bullet & \bullet & \bullet & \bullet & \bullet \\
                \bullet & \bullet & \bullet & \bullet & \bullet \\
                ~~ & \bullet & \bullet & \bullet & \bullet \\
                ~~ & ~~ & ~~ & ~~ & \bullet \\
                ~~ & ~~ & ~~ & ~~ & \bullet \\
                ~~ & ~~ & ~~ & ~~ & \bullet
              \end{array}
$,~~~
$ \mathcal{F}_2= {\color{yellow}{\bullet}} $,~~~
$
\mathcal{F}_1=\begin{array}{ccc}
{\color{blue}{\bullet}} & {\color{blue}{\bullet}} & {\color{blue}{\bullet}} \\
 {\color{blue}{\bullet}} & {\color{blue}{\bullet}} & {\color{blue}{\bullet}} \\
 & {\color{blue}{\bullet}} & {\color{blue}{\bullet}}
              \end{array}
$,~~~
$
\mathcal{F}_3=\begin{array}{c}
                  {\color{green}{\bullet}} \\
                  {\color{green}{\bullet}} \\
                  {\color{green}{\bullet}}
              \end{array}
$.

\end{center}
Then take a proper combination $\mathcal{F}_{12}$ of $ \mathcal{F}_1 $ and $ \mathcal{F}_2 $ on mappings $\phi_1$ and $\phi_2$ as follows
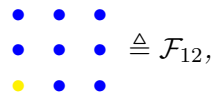
\begin{center}
$ \begin{array}{ccc}
{\color{blue}{\bullet}} & {\color{blue}{\bullet}} & {\color{blue}{\bullet}} \\
 {\color{blue}{\bullet}} & {\color{blue}{\bullet}} & {\color{blue}{\bullet}} \\
{\color{yellow}{\bullet}} & {\color{blue}{\bullet}} & {\color{blue}{\bullet}}
              \end{array} \triangleq \mathcal{F}_{12}
$,
\end{center}
where $\phi_1:\mathcal{F}_1\longrightarrow\mathcal{F}_{12}$ satisfies $\phi_1(i,j)=(i,j)$ for any $(i,j)\in ([3]\times [3])\setminus \{(2,0)\}$, and $\phi_2:\mathcal{F}_2\longrightarrow\mathcal{F}_{12}$ satisfies $\phi_2(0,0)=(2,0)$. Now construct a new Ferrers diagram
\begin{center}
  $\mathcal F^*=\left(
                       \begin{array}{cc}
                         \mathcal F_{12} & \mathcal F_4 \\
                          & \mathcal F_3 \\
                       \end{array}
                     \right)$.
\end{center}
For any prime power $q$, by Theorem $\ref{thm:combine generalized}$, we have an $[\mathcal F^*, 10, 4]_q$ code $\mathcal C^*$ satisfying that for any $\textbf{D}\in {\cal C}^*$, $\textbf{D}|_{\mathcal F_{12}}=\boldsymbol{O}$ if and only if $\textbf{D}|_{\mathcal F_3}=\boldsymbol{O}$, where an optimal $[\mathcal{F}_{12}, 3, 3]_q$ code ${\cal C}_{12}$ exists by Theorem $\ref{thm:shortening}$, an optimal $[\mathcal{F}_4, 7, 4]_q$ code ${\cal C}_4$ exists by Theorem $\ref{thm:from subcodes}$, and an optimal $[\mathcal{F}_3, 3, 1]_q$ code ${\cal C}_3$ is trivial.

The above procedure from $\mathcal F$ to $\mathcal F^*$ yields a natural bijection $\psi$ from $\cal F$ to $\mathcal F^*$ $($note that $\psi(3,3)=(2,0))$. For each $\boldsymbol{D}\in \mathcal C^*$, construct a $10 \times 8$ matrix $\boldsymbol{C_D}$  such that
\begin{equation}
\label{}\nonumber
\boldsymbol{C_D}(i,j)=\left \{
\begin {aligned}
&\boldsymbol{D}({\psi(i,j)}),~~{\rm if}\ (i,j) \in \mathcal F;\\
&0,~~{\rm otherwise}.
\end {aligned}
\right.
\end{equation}
Let $\mathcal C=\{\boldsymbol{C_D}:\boldsymbol{D}\in \mathcal C^*\}$. Then $\mathcal C$ is an optimal $[\mathcal F, 10, 4]_q$ code. Clearly $\cal C$ and $\cal C^*$ have the same linearity and dimension. The optimality is guaranteed by Lemma $\ref{lem:upper bound}$. It suffices to prove that for each nonzero $\boldsymbol{C_D} \in \mathcal C$, rank$(\boldsymbol{C_D})\geq 4$.

Clearly, rank$(\boldsymbol{C_D})\geq$ rank$(\boldsymbol{C_D}|_{\mathcal F_1})+$rank$(\boldsymbol{C_D}|_{\mathcal F_2})+$rank$(\boldsymbol{C_D}|_{\mathcal F_3})$. By Lemma $\ref{lem:prop comb}$, rank$(\boldsymbol{C_D}|_{\mathcal F_1})+$ rank$(\boldsymbol{C_D}|_{\mathcal F_2})\geq $ rank$(\boldsymbol{D}|_{{\mathcal F}_{12}})$. Since rank$(\boldsymbol{C_D}|_{\mathcal F_3})=$ rank$(\boldsymbol{D}|_{\mathcal F_3})$, we have that rank$(\boldsymbol{C_D})\geq$ rank$(\boldsymbol{D}|_{{\mathcal F}_{12}})+$ rank$(\boldsymbol{D}|_{\mathcal F_3})$. Note that $\textbf{D}|_{\mathcal F_{12}}=\boldsymbol{O}$ if and only if $\textbf{D}|_{\mathcal F_3}=\boldsymbol{O}$. If $\textbf{D}|_{\mathcal F_{12}}\neq \boldsymbol{O}$, then since $\mathcal C_{12}$ is an $[\mathcal F_{12},3,3]_q$ code, rank$(\boldsymbol{D}|_{\mathcal F_{12}})\geq 3$, and since $\mathcal C_3$ is an $[\mathcal F_3,3,1]_q$ code, rank$(\boldsymbol{D}|_{\mathcal F_3})\geq 1$. So rank$(\boldsymbol{C_D})\geq 4$. If $\textbf{D}|_{\mathcal F_{12}}=\boldsymbol{O}$, then rank$(\boldsymbol{C_D})=$ rank$(\boldsymbol{C_D}|_{\mathcal F_4})=$ rank$(\boldsymbol{D}|_{\mathcal F_4})$. Since $\mathcal C_4$ is an $[\mathcal F_4,7,4]_q$ code, rank$(\boldsymbol{D}|_{\mathcal F_4})\geq4$. Therefore, rank$(\boldsymbol{C_D})\geq 4$.
\end{example}

\begin{remark}
No known construction can be applied to obtain Example $\ref{eg:com-1}$. By Remark $\ref{rek:subcodes}$, all theorems in Section $2.3.1$ are invalid. Theorem $\ref{thm:from MDS}$ can provide an $[\mathcal F,9,4]_q$ code for any prime power $q\geq 4$, but it is not optimal. Theorem $\ref{thm:combine with same dim}$ is also invalid. Otherwise, to get the required dimension $10$, $\mathcal F_2$ must contain the $7\times 5$ Ferrers diagram in the lower right corner of $\cal F$ as its sub-Ferrers diagram. But then one cannot find $\mathcal F_1$ with at least $10$ dots.
\end{remark}

Let us now generalize Example \ref{eg:com-1}.

\begin{construction}\label{con:com-1}
Let
\[\mathcal F=
\begin{array}{c@{\hspace{-5pt}}c@{\hspace{-5pt}}c}
&\begin{array}{cc}
\overbrace{\rule{15mm}{0mm}}^{n_1}&
\overbrace{\rule{33mm}{0mm}}^{n_4}
\end{array}
\\
\begin{array}{l}
m_1\left.\rule{0mm}{8mm}\right\{\\
\\
\\
\\
\\
\\
\\
\\
\end{array}
&
\begin{array}{ccccccccc}
\bullet & \ldots & \bullet & \bullet & \ldots & \bullet & \bullet & \ldots & \bullet \\
   \vdots & \mathcal F_1 & \vdots & \vdots &  &  &  & \mathcal F_4 & \vdots \\
   \circ & \ldots & \bullet & \bullet & \ldots & \bullet & \bullet &  & \bullet \\
        &  &  & \circ & \ldots & \circ & \bullet &  & \bullet \\
    &  &  & \vdots & \mathcal F_2 & \vdots & \vdots &  & \vdots \\
    &  &  & \circ & \ldots & \circ & \bullet & \ldots & \bullet \\
    &  &  &  &  &  & \circ & \ldots & \bullet \\
    &  &  &  &  &  & \vdots & \mathcal F_3 & \vdots \\
    &  &  &  &  &  & \circ & \ldots & \bullet
\end{array}
& \begin{array}{l}
\left.\rule{0mm}{15mm}\right\}m_4\\
\\\left.\rule{0mm}{8mm}\right\}m_3
\end{array}
\end{array}
\]
be an $m \times n$ Ferrers diagram, where $\mathcal F_i$ is an $m_i \times n_i$ Ferrers sub-diagram, $1\leq i\leq4$,
satisfying that $m=m_3+m_4$, $n=n_1+n_4$, $m_4\geq m_1+m_2$ and $n_4\geq n_2+n_3$. Suppose that 
$\mathcal F_{12}$ is a proper combination of $\mathcal F_1$ and $\mathcal F_2$, and $\mathcal C_{12}$ is an $[\mathcal F_{12}, k_1, \delta_1]_q $ code. If there exist an $[\mathcal F_3, k_3, \delta_3]_q$ code $\mathcal C_3$ and an $[\mathcal F_4, k_4, \delta_4]_q$ code $\mathcal C_4$, then there exists an $[\mathcal F, k, \delta]_q$ code $\mathcal C$, where $k=\min\{k_1,k_3\}+k_4,\delta=\min\{\delta_1 + \delta_3 , \delta_4\}$. Note that the dots $``\bullet"$ in $\cal F$ have to exist, whereas the dots $``\circ"$ can exist or not.
\end{construction}

\begin{proof}
Construct a new Ferrers diagram
\begin{center}
  $\mathcal F^*=\left(
                       \begin{array}{cc}
                         \mathcal F_{12} & \hat{\mathcal{F}_4} \\
                          & \mathcal F_3 \\
                       \end{array}
                     \right)$,
\end{center}
where $\hat{\mathcal{F}_4}$ is obtained by adding the fewest number of new dots to the lower-left corner of ${\cal F}_{4}$ such that ${\cal F}^*$ is a Ferrers diagram. Obviously, an $[\hat{\mathcal{F}_4}, k_4, \delta_4]_q$ code can be obtained by adding all-zero rows to matrices in ${\cal C}_4$. It follows that by Theorem \ref{thm:combine generalized}, an $[\mathcal F^*, k, \delta]_q$ code $\mathcal C^*$ exists, where $k=\min\{k_1,k_3\}+k_4,\delta=\min\{\delta_1 + \delta_3 , \delta_4\}$.

The above procedure from $\mathcal F$ to $\mathcal F^*$ yields a natural injection $\psi$ from $\mathcal F$ to $\mathcal F^*$. Now,  for each $\boldsymbol{D}\in \mathcal C^*$, construct an $m \times n$ matrix $\boldsymbol{C_D}$ such that
\begin{equation}
\label{}\nonumber
\boldsymbol{C_D}(i,j)=\left \{
\begin {aligned}
&\boldsymbol{D}({\psi(i,j)}),~~{\rm if}\ (i,j) \in \mathcal F;\\
&0,~~{\rm otherwise}.
\end {aligned}
\right.
\end{equation}
Let $\mathcal C=\{\boldsymbol{C_D}:\boldsymbol{D}\in \mathcal C^*\}$. It is readily checked that $\mathcal C$ is an $[\mathcal F,k,\delta]_q$ code. \qed
\end{proof}

When ${\cal F}_2$ is empty, Construction \ref{con:com-1} yields Theorem \ref{thm:combine generalized}.

\subsection{Relaxation of dimensions}

Construction \ref{con:com-1} produces an $[\mathcal F,\min\{k_1,k_3\}+k_4,\delta]_q$ code from an $[\mathcal F_{12},k_1,\delta_1]_q$ code $\mathcal C_{12}$, an $[\mathcal F_3,k_3,\delta_3]_q$ code $\mathcal C_3$ and an $[\mathcal F_4,k_4,\delta_4]_q$ code $\mathcal C_4$. This procedure doesn't make full use of dimensions of $\mathcal C_{12}$ and $\mathcal C_3$. We hope to find a proper combination $\mathcal F^{**}$ of $\mathcal F_{12}$ and $\mathcal F_3$ such that there exists an $[\mathcal F^{**},k',\delta']_q$ code, where $k'>\min\{k_1,k_3\}$. The following construction provides a possible way to realize our idea. We start from an example.

\begin{example}\label{eg:com-2}
We here construct an optimal $[\mathcal F,13,4]_q$ code $\mathcal C$ for any prime power $q$, where
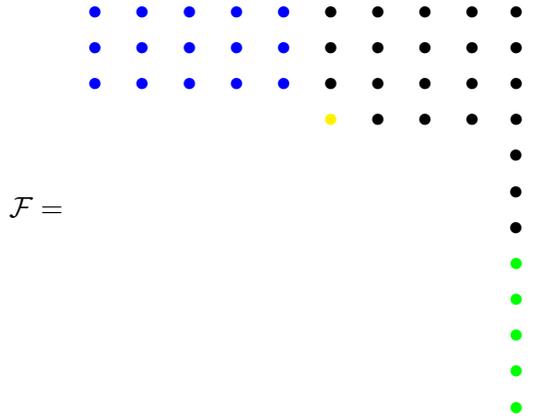
\begin{center}
$\mathcal{F}=
\begin{array}{ccccccccccc}
 {\color{blue}{\bullet}} & {\color{blue}{\bullet}} & {\color{blue}{\bullet}} & {\color{blue}{\bullet}} & {\color{blue}{\bullet}} & \bullet & \bullet & \bullet & \bullet & \bullet \\
 {\color{blue}{\bullet}} & {\color{blue}{\bullet}} & {\color{blue}{\bullet}} & {\color{blue}{\bullet}} & {\color{blue}{\bullet}} & \bullet & \bullet & \bullet & \bullet & \bullet \\
{\color{blue}{\bullet}} & {\color{blue}{\bullet}} & {\color{blue}{\bullet}} & {\color{blue}{\bullet}} & {\color{blue}{\bullet}} & \bullet & \bullet & \bullet & \bullet & \bullet \\
  ~~ & ~~& ~~ & ~~& ~~ & {\color{yellow}{\bullet}} & \bullet & \bullet & \bullet & \bullet \\
  ~~ & ~~ & ~ & ~~ & ~~ & ~~ & ~~ & ~~ & ~~ & \bullet \\
  ~~ & ~~ & ~ & ~~ & ~~ & ~~ & ~~ & ~~ & ~~ & \bullet \\
 ~~ & ~~ & ~ & ~~ & ~~ & ~~ & ~~ & ~~ & ~~ & \bullet \\
 ~~ & ~~ & ~ & ~~ & ~~ & ~~ & ~~ & ~~ & ~~ & {\color{green}{\bullet}} \\
  ~~ & ~~ & ~ & ~~ & ~~ & ~~ & ~~ & ~~ & ~~ & {\color{green}{\bullet}} \\
 ~~ & ~~ & ~ & ~~ & ~~ & ~~ & ~~ & ~~ & ~~ & {\color{green}{\bullet}} \\
 ~~ & ~~ & ~ & ~~ & ~~ & ~~ & ~~ & ~~ & ~~ & {\color{green}{\bullet}}\\
 ~~ & ~~ & ~ & ~~ & ~~ & ~~ & ~~ & ~~ & ~~ & {\color{green}{\bullet}}
\end{array}$.
\end{center}
First take the following four Ferrers sub-diagrams of $\mathcal F$:
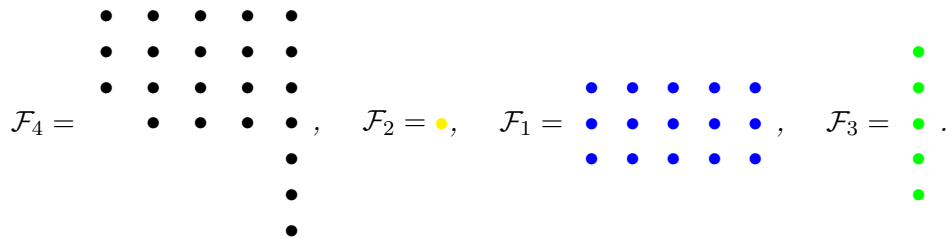
\begin{center}
$
\mathcal{F}_4=\begin{array}{ccccc}
                \bullet & \bullet & \bullet & \bullet & \bullet \\
                \bullet & \bullet & \bullet & \bullet & \bullet \\
                \bullet & \bullet & \bullet & \bullet & \bullet \\
                ~~ & \bullet & \bullet & \bullet & \bullet \\
                ~~ & ~~ & ~~ & ~~ & \bullet \\
                ~~ & ~~ & ~~ & ~~ & \bullet \\
                ~~ & ~~ & ~~ & ~~ & \bullet
              \end{array}
$,~~~
$ \mathcal{F}_2= {\color{yellow}{\bullet}} $,~~~
$
\mathcal{F}_1=\begin{array}{cccccc}
 {\color{blue}{\bullet}} & {\color{blue}{\bullet}} & {\color{blue}{\bullet}} & {\color{blue}{\bullet}} & {\color{blue}{\bullet}} \\
 {\color{blue}{\bullet}} & {\color{blue}{\bullet}} & {\color{blue}{\bullet}} & {\color{blue}{\bullet}} & {\color{blue}{\bullet}} \\
 {\color{blue}{\bullet}} & {\color{blue}{\bullet}} & {\color{blue}{\bullet}} & {\color{blue}{\bullet}} & {\color{blue}{\bullet}}
              \end{array}
$,~~~
$
\mathcal{F}_3=\begin{array}{c}
                  {\color{green}{\bullet}} \\
                  {\color{green}{\bullet}} \\
                  {\color{green}{\bullet}}\\
                  {\color{green}{\bullet}} \\
                  {\color{green}{\bullet}}
              \end{array}
$.

\end{center}
Then take a proper combination $\mathcal{F}_{123}$ of $ \mathcal{F}_1 $, $\mathcal{F}_2 $ and $\mathcal{F}_3$ on mappings $\phi_1$, $\phi_2$ and $\phi_3$ as follows
\begin{center}
$ \begin{array}{ccccc}
 {\color{blue}{\bullet}} & {\color{blue}{\bullet}} & {\color{blue}{\bullet}} & {\color{blue}{\bullet}} & {\color{blue}{\bullet}} \\
 {\color{blue}{\bullet}} & {\color{blue}{\bullet}} & {\color{blue}{\bullet}} & {\color{blue}{\bullet}} & {\color{blue}{\bullet}} \\
 {\color{blue}{\bullet}} & {\color{blue}{\bullet}} & {\color{blue}{\bullet}} & {\color{blue}{\bullet}} & {\color{blue}{\bullet}} \\
 {\color{green}{\bullet}} & {\color{green}{\bullet}} & {\color{green}{\bullet}} & {\color{green}{\bullet}} & {\color{green}{\bullet}} \\
  &  &  &  & {\color{yellow}{\bullet}}
              \end{array} \triangleq \mathcal{F}_{123},
$
\end{center}
where $\phi_1:\mathcal{F}_1\longrightarrow\mathcal{F}_{123}$ satisfies $\phi_1(i,j)=(i,j)$ for any $(i,j)\in [3]\times [5]$, $\phi_2:\mathcal{F}_2\longrightarrow\mathcal{F}_{123}$ satisfies $\phi_2(0,0)=(4,4)$, and $\phi_3:\mathcal{F}_3\longrightarrow\mathcal{F}_{123}$ satisfies $\phi_3(j,0)=(3,j)$ for any $j\in[5]$. By Theorem $\ref{thm:shortening}$, there exists an optimal $[\mathcal{F}_{123}^{t}, 6, 4]_q$ code, which implies an optimal $[\mathcal{F}_{123}, 6, 4]_q$ code $\mathcal C_{123}$. By Theorem $\ref{thm:from subcodes}$, there exists an optimal $[\mathcal{F}_4, 7, 4]_q$ code $\mathcal C_4$.

The above procedure yields two natural bijection $\psi_1:\mathcal F|_{\mathcal F_1,\mathcal F_2,\mathcal F_3}\longrightarrow \mathcal F_{123}$ and $\psi_2:\mathcal F|_{\mathcal F_4}\longrightarrow \mathcal F_4$. For each $\boldsymbol{B}\in \mathcal C_{123}$ and $\boldsymbol{D}\in \mathcal C_{4}$, construct a $12 \times 10$ matrix $\boldsymbol{C_{B,D}}$  such that
\begin{equation}
\label{}\nonumber
\boldsymbol{C_{B,D}}(i,j)=\left \{
\begin {aligned}
&\boldsymbol{B}({\psi_1(i,j)}),~~{\rm if}~(i,j) \in \mathcal F|_{\mathcal F_1,\mathcal F_2,\mathcal F_3};\\
&\boldsymbol{D}({\psi_2(i,j)}),~~{\rm if}~(i,j) \in \mathcal F|_{\mathcal F_4};\\
&0,~~{\rm if}~(i,j) \notin  \mathcal F.
\end {aligned}
\right.
\end{equation}
Let $\mathcal C=\{\boldsymbol{C_{B,D}}:\boldsymbol{B}\in {\cal C}_{123},\boldsymbol{D}\in \mathcal C_4\}$. Then $\mathcal C$ is an optimal $[\mathcal F, 13, 4]_q$ code. Clearly $\cal C$ is a code in $\cal F$ of  dimension $13$. The optimality is guaranteed by Lemma $\ref{lem:upper bound}$. It suffices to prove that for each nonzero $\boldsymbol{C_{B,D}} \in \mathcal C$, rank$(\boldsymbol{C_{B,D}})\geq 4$.

Clearly, rank$(\boldsymbol{C_{B,D}})\geq$ rank$(\boldsymbol{C_{B,D}}|_{\mathcal F_1})+$ rank$(\boldsymbol{C_{B,D}}|_{\mathcal F_2})+$ rank$(\boldsymbol{C_{B,D}}|_{\mathcal F_3})\geq$ rank$(\boldsymbol{B})$, where the second inequality comes from Lemma $\ref{lem:prop comb}$. If $\boldsymbol{B}\neq \boldsymbol{O}$, then since $\mathcal C_{123}$ is an $[\mathcal F_{123},6,4]_q$ code, rank$(\boldsymbol{B})\geq 4$. If $\boldsymbol{B}=\boldsymbol{O}$, then rank$(\boldsymbol{C_{B,D}})=$ rank$(\boldsymbol{C_{B,D}}|_{\mathcal F_4})=$ rank$(\boldsymbol{D})$. Since $\mathcal C_4$ is an $[\mathcal F_4,7,4]_q$ code and $\boldsymbol{D}\neq 0$, we have rank$(\boldsymbol{D})\geq 4$. Therefore, rank$(\boldsymbol{C_{B,D}})\geq 4$.
\end{example}

\begin{remark}
No known construction can be applied to obtain Example $\ref{eg:com-2}$. By Remark $\ref{rek:subcodes}$, all theorems in Section $2.3.1$ are invalid. Theorem $\ref{thm:from MDS}$ can provide an $[\mathcal F,11,4]_q$ code for any prime power $q\geq 4$, but it is not optimal. Theorem $\ref{thm:combine with same dim}$ is also invalid. Otherwise, to get the required dimension $13$, $\mathcal F_2$ must be the $9\times 5$ Ferrers diagram in the lower right corner of $\cal F$, and $\mathcal F_1$ must be the $3\times 5$ Ferrers diagram in the top left corner of $\cal F$. Then $\delta_1=\delta_2=1$ because of the dimension $13$, which contradicts with $\delta=4$.
\end{remark}

Let us now generalize Example \ref{eg:com-2}.

\begin{construction}\label{con:com-2}
Let
\[\mathcal F=
\begin{array}{c@{\hspace{-5pt}}c@{\hspace{-5pt}}c}
&\begin{array}{cc}
\overbrace{\rule{15mm}{0mm}}^{n_1}&
\overbrace{\rule{33mm}{0mm}}^{n_4}
\end{array}
\\
\begin{array}{l}
m_1\left.\rule{0mm}{8mm}\right\{\\
\\
\\
\\
\\
\\
\\
\\
\end{array}
&
\begin{array}{ccccccccc}
\bullet & \ldots & \bullet & \bullet & \ldots & \bullet & \bullet & \ldots & \bullet \\
   \vdots & \mathcal F_1 & \vdots & \vdots &  &  &  & \mathcal F_4 & \vdots \\
   \circ & \ldots & \bullet & \bullet & \ldots & \bullet & \bullet &  & \bullet \\
        &  &  & \circ & \ldots & \circ & \bullet &  & \bullet \\
    &  &  & \vdots & \mathcal F_2 & \vdots & \vdots &  & \vdots \\
    &  &  & \circ & \ldots & \circ & \bullet & \ldots & \bullet \\
    &  &  &  &  &  & \circ & \ldots & \bullet \\
    &  &  &  &  &  & \vdots & \mathcal F_3 & \vdots \\
    &  &  &  &  &  & \circ & \ldots & \bullet
\end{array}
& \begin{array}{l}
\left.\rule{0mm}{15mm}\right\}m_4\\
\\\left.\rule{0mm}{8mm}\right\}m_3
\end{array}
\end{array}
\]
be an $m \times n$ Ferrers diagram, where $\mathcal F_i$ is an $m_i \times n_i$ Ferrers sub-diagram, $1\leq i\leq4$, satisfying that $m=m_3+m_4$, $n=n_1+n_4$, $m_4\geq m_1+m_2$ and $n_4\geq n_2+n_3$.
Suppose that
$\mathcal F_{123}$ is a proper combination of $\mathcal F_1$, $\mathcal F_2$ and $\mathcal F_3$, and $\mathcal C_{123}$ is an $[\mathcal F_{123}, k_1, \delta_1]_q$ code. If there exists an $[\mathcal F_4, k_4, \delta_4]_q$ code $\mathcal C_4$, then there exists an $[\mathcal F, k_1+k_4, \delta]_q$ code $\mathcal C$, where $\delta=\min\{\delta_1 , \delta_4\}$.
\end{construction}

\begin{proof}
Take two natural bijections $\psi_1:\mathcal F|_{\mathcal F_1,\mathcal F_2,\mathcal F_3}\longrightarrow \mathcal F_{123}$ and $\psi_2:\mathcal F|_{\mathcal F_4}\longrightarrow \mathcal F_4$. For each $\boldsymbol{B}\in \mathcal C_{123}$ and $\boldsymbol{D}\in \mathcal C_{4}$, construct an $m \times n$ matrix $\boldsymbol{C_{B,D}}$  such that
\begin{equation}
\label{}\nonumber
\boldsymbol{C_{B,D}}(i,j)=\left \{
\begin {aligned}
&\boldsymbol{B}({\psi_1(i,j)}),~~{\rm if}~(i,j) \in \mathcal F|_{\mathcal F_1,\mathcal F_2,\mathcal F_3};\\
&\boldsymbol{D}({\psi_2(i,j)}),~~{\rm if}~(i,j) \in \mathcal F|_{\mathcal F_4};\\
&0,~~{\rm if}~(i,j) \notin  \mathcal F.
\end {aligned}
\right.
\end{equation}
Let $\mathcal C=\{\boldsymbol{C_{B,D}}:\boldsymbol{B}\in {\cal C}_{123},\boldsymbol{D}\in \mathcal C_4\}$. It is readily checked that $\mathcal C$ is an $[\mathcal F, k_1+k_4, \delta]_q$ code $\mathcal C$, where $\delta=\min\{\delta_1, \delta_4\}$. \qed
\end{proof}

\begin{remark}
Compared with Construction $\ref{con:com-1}$, Construction $\ref{con:com-2}$ starts from a proper combination $\mathcal F_{123}$ of $\mathcal F_1$, $\mathcal F_2$ and $\mathcal F_3$, which must be a Ferrers diagram according to the definition of proper combinations. This requirement sometimes restricts the value of rank of the resulting code $\mathcal C$. For example, in Example $\ref{eg:com-1}$, the proper combination of $\mathcal F_{12}$ and $\mathcal F_3$ will provide codes with rank at most $3$, while the required code has rank $4$.
\end{remark}

\begin{theorem}\label{thm:com-1}
Let $\delta\leq y\leq$ min$\{m-\delta+2,n-\delta-1\}$. Let
\[\mathcal F=
\begin{array}{c@{\hspace{-1pt}}c@{\hspace{-1pt}}c}
&\begin{array}{ccc}
\overbrace{\rule{16mm}{0mm}}^{n-y}&
\overbrace{\rule{16mm}{0mm}}^{y-\delta+1}&
\overbrace{\rule{16mm}{0mm}}^{\delta-1}
\end{array}
\\
&
\begin{array}{cccccccccc}
~{\color{black}{\bullet}} & \cdots & {\color{black}{\bullet}} & \bullet & \cdots & \bullet & \bullet & \cdots & \bullet \\
 ~\vdots & \ddots & \vdots & \vdots & \vdots & \vdots & \vdots & \ddots & \vdots   \\
 ~{\color{black}{\bullet}}& \cdots & {\color{black}{\bullet}} & \bullet & \cdots & \bullet & \bullet & \cdots & \bullet  \\
  &  &  & \circ & \cdots & \circ & \bullet & \ldots & \bullet   \\
  &  &  & \vdots & \mathcal P & \vdots &  \vdots & \ddots  & \vdots   \\
  &  &  & \circ & \cdots & \circ & \bullet & \cdots & \bullet \\
  &  &  &  &  &  &  &  & \bullet    \\
  &  &  &  &  &  &  &  & \vdots   \\
  &  &  &  &  &  &  &  & \bullet \\
  &  &  &  &  &  &  &  & {\color{black}{\bullet}}  \\
  &  &  &  &  &  &  &  & \vdots   \\
  &  &  &  &  &  &  &  & {\color{black}{\bullet}}
\end{array}
& \begin{array}{l}
\left.\rule{0mm}{7.5mm}\right\}\delta-1\vspace{0.3cm}\\
\left.\rule{0mm}{7.5mm}\right\}y-\delta\vspace{0.3cm}\\
\left.\rule{0mm}{8mm}\right\}\delta-1\vspace{0.3cm}\\
\left.\rule{0mm}{8mm}\right\}m-y-\delta+2
\end{array}
\end{array}
\]
be an $m \times n$ Ferrers diagram $ \mathcal F$. Let $z_i$ be the number of dots in the $i$-th column of $\mathcal P$, $i\in[y-\delta+1]$. If $z_0\leq n-y$, then Construction $\ref{con:com-2}$ provides an optimal $[\mathcal F, k, \delta]_q$ code, where
\begin{equation}
\label{}\nonumber
k=\left \{
\begin {aligned}
&m-y+1+(y-\delta)(\delta-1)+|\mathcal P|,~~{\rm if}~m-n\leq \delta-2;\\
&n-1+(y-\delta)(\delta-2)+|\mathcal P|,~~otherwise.\\
\end {aligned}
\right.
\end{equation}
\end{theorem}
\begin{proof}
Let $\mathcal P_1$ denote the Ferrers diagram obtained by removing the first column of $\mathcal P$.
 Consider the following four Ferrers sub-diagrams of $\mathcal F$:
\begin{center}
$\mathcal F_1=\begin{array}{c@{\hspace{-5pt}}c@{\hspace{-5pt}}c}
&\begin{array}{cc}
\overbrace{\rule{15mm}{0mm}}^{n-y}
\end{array}
\\
&
\begin{array}{ccc}
               {\color{black}{\bullet}} & \cdots & {\color{black}{\bullet}}  \\
               \vdots & \ddots & \vdots  \\
               {\color{black}{\bullet}} & \cdots & {\color{black}{\bullet}}
             \end{array}
& \begin{array}{l}
\left.\rule{0mm}{8mm}\right\}\delta-1
\end{array}
\end{array},~~~~~~~
{\cal F}_2=
\begin{array}{c@{\hspace{-1pt}}c@{\hspace{-1pt}}c}
\begin{array}{c}
\circ \\
\vdots \\
\circ
\end{array} &
\begin{array}{l}
\left.\rule{0mm}{8mm}\right\}z_0
\end{array}
\end{array},~~~~~~~
{\mathcal F}_3=\begin{array}{c@{\hspace{-1pt}}c@{\hspace{-1pt}}c}
\begin{array}{c}
\bullet  \\
\vdots  \\
\bullet
\end{array}
& \begin{array}{l}
\left.\rule{0mm}{8mm}\right\}m-y-\delta+2
\end{array}
\end{array}$,
\end{center}

\begin{center}
$\mathcal F_4=
\begin{array}{c@{\hspace{-1pt}}c@{\hspace{-1pt}}c}
&\begin{array}{cc}
\overbrace{\rule{22mm}{0mm}}^{y-\delta+1}&
\overbrace{\rule{20mm}{0mm}}^{\delta-1}
\end{array}
\\
&
\begin{array}{cccccccc}
               \bullet & \bullet & \cdots & \bullet & \bullet & ~\cdots~ & \bullet \\
               \vdots & \vdots & \ddots & \vdots & \vdots & \ddots & \vdots  \\
               \bullet & \bullet &\cdots & \bullet & \bullet & \cdots & \bullet \\
            & \circ & \cdots & \circ & \bullet & \cdots & \bullet  \\
             & \vdots & \mathcal P_1 & \vdots & \vdots & \ddots & \vdots  \\
            & \circ & \cdots & \circ & \bullet & \cdots & \bullet  \\
             &   &  &  &  &  & \bullet  \\
              &  &  &  &  &  & \vdots  \\
               & &  &  &  &  & \bullet
\end{array}
& \begin{array}{l}
\left.\rule{0mm}{9mm}\right\}\delta-1\\
\left.\rule{0mm}{9mm}\right\}y-\delta\\
\left.\rule{0mm}{9mm}\right\}\delta-1
\end{array}
\end{array}
.
\begin{array}{cc}
&
\end{array}$
\end{center}

By Theorem \ref{thm:from subcodes}, there exists an $[\mathcal F_4, (y-\delta+1)(\delta-1)-z_0+|\mathcal P|, \delta]_q$ code.

If $n\geq m-\delta+2$, then $n-y\geq m-\delta+2-y$. When $m-y-\delta+2\geq z_0$, take a proper combination ${\cal F}_{123}$ of $\mathcal{F}_1 $, $ \mathcal{F}_2 $ and $ \mathcal{F}_3$ as follows (note that $z_0\leq n-y$ by assumption)
\begin{center}
  $\mathcal F_{123}=\left(
                       \begin{array}{c}
                         \mathcal F_{1} \\ \mathcal F_3^t \\
                         \mathcal F_2^t \\
                       \end{array}
                     \right)$.
\end{center}
When $m-y-\delta+2<z_0$, take
\begin{center}
  $\mathcal F_{123}=\left(
                       \begin{array}{c}
                         \mathcal F_{1} \\
                         \mathcal F_2^t \\
                         \mathcal F_3^t \\
                       \end{array}
                     \right)$.
\end{center}
${\cal F}_{123}$ is an $(\delta+1)\times(n-y)$ Ferrers diagram. By assumption, $n-y\geq \delta+1$, so by Theorem \ref{thm:shortening}, there exists an $[\mathcal F_{123}, m-y-\delta+2+z_0, \delta]_q$ code.
Then apply Construction \ref{con:com-2} to obtain an $[\mathcal{F}, m-y+1+(y-\delta)(\delta-1)+|\mathcal P|, \delta]_q$ code, which is optimal by Lemma \ref{lem:upper bound} (one can check it by counting the number of dots in $\cal F$ which are not contained in the first $\delta-1$ rows).

If $n< m-\delta+2$, then $n-y<m-\delta+2-y$. Take a proper combination ${\cal F}_{123}$ of $\mathcal{F}_1 $, $ \mathcal{F}_2 $ and $ \mathcal{F}_3$ as follows
\begin{center}
  $\mathcal F_{123}=\left(
                       \begin{array}{c}
                         \mathcal F_3^t \\
                         \mathcal F_{1} \\
                         \mathcal F_2^t \\
                       \end{array}
                     \right)$.
\end{center}
${\cal F}_{123}$ is an $(\delta+1)\times(m-y-\delta+2)$ Ferrers diagram. By assumption, $n-y\geq \delta+1$, so $m-y-\delta+2>\delta+1$. Thus by Theorem \ref{thm:shortening}, there exists an $[\mathcal F_{123}, n-y+z_0, \delta]_q$ code. Then apply Construction \ref{con:com-2} to obtain an $[\mathcal{F}, n-1+(y-\delta)(\delta-2)+|\mathcal P|, \delta]_q$ code, which is optimal by Lemma \ref{lem:upper bound} (one can check it by counting the number of dots in $\cal F$ which are not contained in the first $\delta-2$ rows and the rightmost column). \qed
\end{proof}

We remark that Theorem \ref{thm:com-1} with $n=10$, $y=5$, $m=12$, $\delta=4$, $k=13$ and $\mathcal P= {{\bullet}\ \ {\bullet}}\ $ yields Example \ref{eg:com-2}.

\subsection{A special case: $\mathcal F_2$ having only one dot}

Constructions \ref{con:com-1} and \ref{con:com-2} require that $\mathcal F_2$ doesn't contain any dots of $\mathcal F$ in the first $n_1$ columns and the last $m_3$ rows. However, when $\mathcal F_2$ contains only one dot, this restriction can be relaxed.

\begin{construction} \label{con:com-3}
Let $m=m_1+m_3$ and $n=n_1+n_3$. Let
\[\mathcal F=
\begin{array}{c@{\hspace{-5pt}}c@{\hspace{-5pt}}c}
&\begin{array}{cc}
\overbrace{\rule{15mm}{0mm}}^{n_1}&
\overbrace{\rule{15mm}{0mm}}^{n_3}
\end{array}
\\
&
\begin{array}{cccccc}
     \bullet & \cdots & \bullet & \bullet & \cdots & \bullet \\
     \vdots & \mathcal F_1 & \vdots & \vdots & \mathcal F_4 & \vdots \\
     \circ & \cdots & \bullet & \bullet & \cdots & \bullet \\
      &  & {\color{green}{\bullet}} & \bullet & \cdots & \bullet \\
      &  &  & \vdots & \mathcal F_3 & \vdots \\
      &  &  & \circ & \cdots & \bullet
   \end{array}
& \begin{array}{l}
\left.\rule{0mm}{8mm}\right\}m_1\vspace{0.2cm}\\
\left.\rule{0mm}{8mm}\right\}m_3
\end{array}
\end{array}
\]
be an $m \times n$ Ferrers diagram, where $\mathcal F_1$ is an $m_1 \times n_1$ Ferrers diagram, $\mathcal F_2=\begin{array}{c}{\color{green}{\bullet}}\end{array}$, $\mathcal F_3$ is an $m_3 \times n_3$ Ferrers diagram, and $\mathcal F_4$ is an $m_1 \times n_3$ full Ferrers diagram. Sort the list $\{1\}\cup \{\rho_i(\mathcal F_1):i \in [m_1]\}\cup\{\gamma_j(\mathcal F_3):j \in [n_3]\}$ from small to large, where $\rho_i(\mathcal F_1)$ denotes the number of dots in the $i$-th row of ${\cal F}_1$ and $\gamma_j(\mathcal F_3)$ denotes the number of dots in the $j$-th column of ${\cal F}_3$. The elements in the sorted list are rewritten as $\alpha_0\leq \alpha_1 \leq \ldots \leq \alpha_{m_1+n_3}$.
Suppose that
$\mathcal F_{123}$ is a proper combination of $\mathcal F_1$, $\mathcal F_2$ and $\mathcal F_3$ satisfying $\gamma_l(\mathcal F_{123})=\alpha_{l}$ for $l\in [m_1+n_3+1]$, where $\gamma_l(\mathcal F_{123})$ denotes the number of dots in the $l$-th column of ${\cal F}_{123}$, and ${\cal C}_{123}$ is an $[\mathcal F_{123}, k_1, \delta_1]_q$ code. If there exists an $[\mathcal F_4, k_4, \delta_4]_q$ code ${\cal C}_4$, then there exists an $[\mathcal F, k_1+k_4, \delta]_q$ code $\cal C$, where $\delta=\min\{\delta_1, \delta_4\}$.
\end{construction}

\begin{proof}
Take a natural bijection $\psi_1:\mathcal F|_{\mathcal F_1,\mathcal F_2,\mathcal F_3}\longrightarrow \mathcal F_{123}$ such that $\psi_1(m_1,n_1-1)=(0,0)$, $\psi_1(i,n_1-1)=(0,*)$ for each $i\in[m_1]$, and $\psi_1(m_1,j)=(0,*)$ for each $n_1\leq j\leq n_1+n_3-1$. Take a natural bijection $\psi_2:\mathcal F|_{\mathcal F_4}\longrightarrow \mathcal F_4$. For each $\boldsymbol{B}\in \mathcal C_{123}$ and $\boldsymbol{D}\in \mathcal C_{4}$, construct an $m \times n$ matrix $\boldsymbol{C_{B,D}}$  such that
\begin{equation}
\label{}\nonumber
\boldsymbol{C_{B,D}}(i,j)=\left \{
\begin {aligned}
&\boldsymbol{B}({\psi_1(i,j)}),~~{\rm if}~(i,j) \in \mathcal F|_{\mathcal F_1,\mathcal F_2,\mathcal F_3};\\
&\boldsymbol{D}({\psi_2(i,j)}),~~{\rm if}~(i,j) \in \mathcal F|_{\mathcal F_4};\\
&0,~~{\rm if}~(i,j) \notin  \mathcal F.
\end {aligned}
\right.
\end{equation}
Let $\mathcal C=\{\boldsymbol{C_{B,D}}:\boldsymbol{B}\in {\cal C}_{123},\boldsymbol{D}\in \mathcal C_4\}$. Then $\mathcal C$ is an $[\mathcal F, k_1+k_4, \delta]_q$ code $\mathcal C$, where $\delta=\min\{\delta_1, \delta_4\}$.

One can easily verify the linearity and the dimension of the code. It suffices to examine the minimum rank weight of any nonzero codewords $\boldsymbol{C_{B,D}}$ from $\mathcal C$. We give a sketch of the counting for ranks below. The technique is similar to that in Example \ref{eg:com-2}.

Let
\begin{center}
$\boldsymbol{C_{B,D}}=\left(
         \begin{array}{ccc;{2pt/2pt}c;{2pt/2pt}ccc}
           & &  &  * &   &     \\
        & \boldsymbol{A_{1}} &     & \vdots &  &\boldsymbol{A_{4}} &   \\
        & &    & * &  &    \\ \hdashline[2pt/2pt]
        0 & \cdots & 0 & a  & * & \cdots & *  \\ \hdashline[2pt/2pt]
         & &    & 0  &   &  &   \\
         &   &    & \vdots  &   & \boldsymbol{A_{3}}  & \\
          & &   & 0  &   &   & \\
         \end{array}
       \right)$,
\end{center}
where $a$ corresponds to the dot in ${\cal F}_2$. If $\boldsymbol{A_{4}}\neq \boldsymbol{O}$, then since ${\cal C}_4$ is an $[\mathcal F_4, k_4, \delta_4]_q$ code, rank$(\boldsymbol{C_{B,D}})\geq$ rank$(\boldsymbol{A_{4}})={\rm rank}(\boldsymbol{D})\geq \delta_4$. If $\boldsymbol{A_{4}}= \boldsymbol{O}$ and $a=0$, then since ${\cal C}_{123}$ is an $[\mathcal F_{123}, k_1, \delta_1]_q$ code, rank$(\boldsymbol{C_{B,D}})=$ rank$(\boldsymbol{C_{B,D}}|_{\mathcal F_1})+ {\rm rank}(\boldsymbol{C_{B,D}}|_{\mathcal F_3})\geq {\rm rank}(\textbf{B})\geq \delta_1$. If $\boldsymbol{A_{4}}= \boldsymbol{O}$ and $a\neq 0$, then rank$(\boldsymbol{C_{B,D}})\geq {\rm rank}(\boldsymbol{A_{1}})+1+{\rm rank}(\boldsymbol{A_{3}})$. According to $\psi_1$, $\textbf{B}$ is of the form (a permutation of columns are allowed)
\begin{center}
$\left(
\begin{array}{c;{2pt/2pt}ccc;{2pt/2pt}ccc}
a  & * & \cdots & * & * & \cdots & * \\
   & & \boldsymbol{A_{1}^T} & & & \boldsymbol{A_{3}} \\
\end{array}
\right)$.
\end{center}
Since ${\rm rank}(\boldsymbol{A_{1}})+1+{\rm rank}(\boldsymbol{A_{3}})\geq {\rm rank}(\boldsymbol{B})$, we have rank$(\boldsymbol{C_{B,D}})\geq \delta_1$. \qed
\end{proof}

\begin{theorem} \label{thm:com-3}
Take $\delta_1=\delta_4=\delta$ in Construction $\ref{con:com-3}$ such that $\delta \leq m_1+1$. Suppose that $\mathcal{F}$ in Construction $\ref{con:com-3}$ satisfies:
\begin{itemize}
\item[$(1)$] if $\delta<m_1+1$, then $n_3 \geq m_1;$
\item[$(2)$] $1+m_1+n_3\leq \max\{n_1,m_3\};$
\item[$(3)$] $\alpha_{m_1+n_3-\delta+2}\geq m_1+n_3;$
\item[$(4)$] $\rho_{\delta-2}-n_3\geq m_3$,
\end{itemize}
where $\rho_i$ denotes the number of dots in the $i$-th row of ${\cal F}$, $i\in [m_1+m_3]$. Then there exists an optimal $[\mathcal F, \sum_{i=\delta-1}^{m_1+m_3-1}\rho_i, \delta]_q$ code $\mathcal C$ for any prime power $q$.
\end{theorem}

\begin{proof}
By Theorem \ref{thm:shortening}, due to Condition $(1)$, there is an optimal $[\mathcal F_4,n_3(m_1-\delta+1),\delta]_q$ code ${\cal C}_4$ for any prime power $q$. Note that when $\delta=m_1+1$, it consists of only a zero codeword.

Note that $\mathcal F_{123}$ has $m_1+n_3+1$ columns. By Theorem \ref{thm:from subcodes}, due to Conditions (2) and (3), there is an $[\mathcal F_{123}, \sum_{i=0}^{m_1+n_3-\delta+1} \alpha_i, \delta]_q$ code ${\cal C}_{123}$ for any prime power $q$, where $\alpha_i$ denotes the number of dots in the $i$-th column of $\mathcal F_{123}$. It is optimal by Lemma \ref{lem:upper bound}. Condition (4) ensures all dots in ${\cal F}_3$ contribute dimensions for ${\cal C}_{123}$, so $\sum_{i=0}^{m_1+n_3-\delta+1} \alpha_i=\sum_{i=\delta-1}^{m_1-1}(\rho_i-n_3)+\sum_{i=m_1}^{m_1+m_3-1}\rho_i$.

Therefore, we can apply Construction \ref{con:com-3} to obtain an optimal $[\mathcal F, k, \delta]_q$ code, where $k=n_3(m_1-\delta+1)+\sum_{i=\delta-1}^{m_1-1}(\rho_i-n_3)+\sum_{i=m_1}^{m_1+m_3-1}\rho_i=\sum_{i=\delta-1}^{m_1+m_3-1}\rho_i$. \qed
\end{proof}

\begin{example}\label{eg:com-3}
Consider the following Ferrers diagram:
\begin{center}
$\mathcal F=\begin{array}{ccccc}
   {\color{yellow}{\bullet}}  & {\color{yellow}{\bullet}}  & {\color{yellow}{\bullet}}  & {\color{yellow}{\bullet}}  & \bullet \\
   {\color{yellow}{\bullet}}  & {\color{yellow}{\bullet}}  & {\color{yellow}{\bullet}}  & {\color{yellow}{\bullet}}  & \bullet \\
    &  &  & {\color{red}{\bullet}}  & {\color{green}{\bullet}} \\
    &  &  &  & {\color{green}{\bullet}} \\
    &  &  &  & {\color{green}{\bullet}}\\
    &  &  &  & {\color{green}{\bullet}}
 \end{array}
$.
\end{center}
Let $\delta=3$ and
\begin{center}
$\mathcal F_1=\begin{array}{cccc}
   {\color{yellow}{\bullet}}  & {\color{yellow}{\bullet}}  & {\color{yellow}{\bullet}}  & {\color{yellow}{\bullet}} \\
   {\color{yellow}{\bullet}}  & {\color{yellow}{\bullet}}  & {\color{yellow}{\bullet}}  & {\color{yellow}{\bullet}}
             \end{array},~~
 \mathcal F_2=\begin{array}{c}
                {\color{red}{\bullet}}
              \end{array},~~
  \mathcal F_3=\begin{array}{c}
                 {\color{green}{\bullet}} \\
                 {\color{green}{\bullet}} \\
                 {\color{green}{\bullet}} \\
                 {\color{green}{\bullet}}
               \end{array},~~
                  \mathcal F_4=\begin{array}{c}
                \bullet \\
                \bullet
              \end{array}.$
\end{center}
Then $m_1=2$, $n_1=4$, $m_3=4$, $n_3=1$, $\alpha_0=1$ and $\alpha_i=4$ for $i\in \{1,2,3\}$. So the conditions in  Theorem $\ref{thm:com-3}$ are satisfied, and we can construct an optimal $[\mathcal F, 5,3]$ code.
\end{example}

\begin{remark}
No known construction can be applied to obtain Example $\ref{eg:com-3}$. By Remark $\ref{rek:subcodes}$, all theorems in Section $2.3.1$ are invalid. Theorem $\ref{thm:from MDS}$ provides an $[\mathcal F,4,3]_q$ code for any prime power $q\geq 3$, but it is not optimal. Theorem $\ref{thm:combine with same dim}$ is also invalid. Otherwise, to get the required dimension $4$, $\mathcal F_2$ must be the $4\times 2$ Ferrers diagram in the lower right corner of $\cal F$, and $\mathcal F_1$ must be the $2\times 3$ Ferrers diagram in the top left corner of $\cal F$. Then $\delta_1=2$ and $\delta_2=1$ because of the dimension $4$, which contradicts with $\delta=4$.
\end{remark}

\section{Concluding remarks}

Main contributions of this paper lie in the following two aspects. One is to generalize Construction 2 in \cite{egrw} by exploring subcodes of Gabidulin codes. Construction 2 in \cite{egrw} requires that each of the rightmost $\delta-1$ columns in Ferrers diagram $\cal F$ has at least $n-1$ dots. We relax the condition $n-1$ to $n-r$ (see Theorem \ref{thm:subcodes from Gab}). The other is to generalize Theorem 9 in \cite{egrw} by introducing the concept of proper combinations of Ferrers diagrams (see Constructions \ref{con:com-1}, \ref{con:com-2} and \ref{con:com-3}). This is the first time constructions for FDRM codes with large size based on small ones are investigated systematically since they are introduced in \cite{egrw}.

Recently, a new family of MRD codes is presented in \cite{Sheekey}. A natural question is how to use it to construct new optimal FDRM codes.

Another question is whether it is possible in some circumstances to require that $\mathcal F_1$ and $\mathcal F_2$ in Construction \ref{con:com-1} or $\mathcal F_1$, $\mathcal F_2$ and $\mathcal F_3$ in Construction \ref{con:com-2} are not Ferrers diagrams.

\appendix
\section{Appendix}\label{app-A}

\textbf{Proof of Lemma \ref{lem:sys MRD}}~~Let
\begin{center}
$
 \textbf{B}=\left(
             \begin{array}{llll}
               1 & u_{0,1} & \cdots & u_{0,n-1} \\
                & 1 & \cdots & u_{1,n-1} \\
                &  & \ddots & \vdots \\
                &  &  & 1 \\
             \end{array}
           \right)$.
\end{center}
Then
\begin{center}
{\scriptsize $\textbf{GB}=$}\begin{scriptsize}$\left(
\begin{array}{ccccccc}
  1 & u_{0,1} & \ldots & u_{0,k-1} & u_{0,k}+a_{1,k}\beta^{k} & \cdots & u_{0,n-1}+\sum_{i=k}^{n-2} u_{i,n-1}a_{1,i}\beta^i+a_{1, n-1}\beta^n \\
    & 1 & \ldots & u_{1,k-1} & u_{1,k}+a_{2,k}\beta^{k-1} & \ldots & u_{1,n-1}+\sum_{i=k}^{n-2} u_{i,n-1}a_{2,i}\beta^{i-1}+a_{2,n-1}\beta^{n-2} \\
   &  & \ddots & \vdots & \vdots & \ddots & \vdots \\
    &  &  & 1 & u_{k-1,k}+a_{k,k}\beta & \cdots & u_{k-1,n-1}+\sum_{i=k}^{n-2} u_{i,n-1}a_{k,i}\beta^{i-k+1}+a_{k,n-1}\beta^{n-k} \\
  \end{array}
\right).
 $\end{scriptsize}
\end{center}
Let $\textbf{D}_k$ be any $k \times k$ submatrix of $\textbf{GB}$. Then $\det(\textbf{D}_k)$ is a polynomial on $\beta$.

\textbf{Case $1.$} $\textbf{D}_k$ doesn't contain the last column of $\textbf{GB}$. If we could prove that the degree of $\det(\textbf{D}_k)$ is less than $m$, and the leading coefficient of $\det(\textbf{D}_k)$ is a minor of $\textbf{A}_1$, then since every minor of $\textbf{A}_1$ is nonzero, we would have $\det(\textbf{D}_k)\neq 0$.

\textbf{Subcase $1.1.$} $\textbf{D}_k$ doesn't contain any of the first $k$ columns of $\textbf{GB}$. Take
\begin{center}
$\textbf{M}_1=\left(
            \begin{array}{cccc}
              a_{1,i_1}\beta^{i_1} & a_{1,i_2}\beta^{i_2} & \cdots & a_{1,i_k}\beta^{i_k} \\
              a_{2,i_1}\beta^{i_1-1} & a_{1,i_2}\beta^{i_2-1} & \cdots & a_{2,i_k}\beta^{i_k-1} \\
              \vdots & \vdots & \ddots & \vdots \\
              a_{k,i_1}\beta^{i_1-k+1} & a_{k,i_2}\beta^{i_2-k+1} & \cdots & a_{k,i_k}\beta^{i_k-k+1} \\
            \end{array}
          \right)$,
 \end{center}
where $\{i_1,i_2,\ldots,i_k\}\subseteq\{k,k+1,\ldots,n-2\}$, such that the degree of $\det(\textbf{M}_1)$ is the same as that of $\det(\textbf{D}_k)$, and their leading coefficients are the same. Then
\begin{center}\small
${\rm det}(\textbf{M}_1)=\det\left(
            \begin{array}{cccc}
              a_{1,i_1}\beta^{k-1} & a_{1,i_2}\beta^{k-1} & \cdots & a_{1,i_k}\beta^{k-1} \\
              a_{2,i_1}\beta^{k-2} & a_{2,i_2}\beta^{k-2} & \cdots & a_{2,i_k}\beta^{k-2} \\
              \vdots & \vdots & \ddots & \vdots \\
              a_{k,i_1} & a_{k,i_2} & \cdots & a_{k,i_k} \\
            \end{array}
          \right)\cdot(\beta^{i_1-k+1}\beta^{i_2-k+1}\cdots\beta^{i_k-k+1})
          =\det\left(
            \begin{array}{cccc}
              a_{1,i_1} & a_{1,i_2} & \cdots & a_{1,i_k} \\
              a_{2,i_1} & a_{2,i_2} & \cdots & a_{2,i_k} \\
              \vdots & \vdots & \ddots & \vdots \\
              a_{k,i_1} & a_{k,i_2} & \cdots & a_{k,i_k} \\
            \end{array}
          \right)\cdot(\beta^{k-1}\beta^{k-2}\cdots\beta)\cdot(\beta^{i_1-k+1}\beta^{i_2-k+1}\cdots\beta^{i_k-k+1})
 $,
\end{center}
whose degree is $k(k-1)/2+\sum_{j=1}^k (i_j-k+1)\leq kn-k^2-k<m$. Since every $k$-minor of $\textbf{A}_1$ is nonzero, the leading coefficient of $\det(\textbf{M}_1)$ is nonzero. So $\det(\textbf{D}_k)\neq 0$.

\textbf{Subcase $1.2.$} $\textbf{D}_k$ contains $h$ columns coming from the first $k$ columns of $\textbf{GB}$ for some $1\leq h\leq k$. Write these $h$ columns as the $j_1$-th, $j_2$-th, $\ldots$, $j_h$-th columns. Let $\textbf{U}_{k\times h}$ be the submatrix formed by the first $h$ columns of $\textbf{D}_k$. Take
\begin{center}
 $\textbf{M}_2=\left(
                          \begin{array}{c;{2pt/2pt}cccc}
                             & a_{1,i_{h+1}}\beta^{i_{h+1}} & a_{1,i_{h+2}}\beta^{i_{h+2}} & \cdots & a_{1,i_{k}}\beta^{i_{k}} \\
                            \textbf{U}_{k\times h} & a_{2,i_{h+1}}\beta^{i_{h+1}-1} & a_{2,i_{h+2}}\beta^{i_{h+2}-1} & \cdots & a_{2,i_{k}}\beta^{i_{k}-1} \\
                             & \vdots & \vdots & \ddots & \vdots \\
                             & a_{k,i_{h+1}}\beta^{i_{h+1}-k+1} & a_{k,i_{h+2}}\beta^{i_{h+2}-k+1} & \cdots & a_{k,i_{k}}\beta^{i_{k}-k+1} \\
                          \end{array}
                        \right)$,
\end{center}
where $\{i_{h+1},i_{h+2},\ldots,i_k\}\subseteq\{k,k+1,\ldots,n-2\}$, such that the degree of $\det(\textbf{M}_2)$ is the same as that of $\det(\textbf{D}_k)$, and their leading coefficients are the same. Then
\begin{center}\tiny
$\det(\textbf{M}_2)=\det\left(
                          \begin{array}{c;{2pt/2pt}cccc}
                             & a_{1,i_{h+1}}\beta^{k-1} & a_{1,i_{h+2}}\beta^{k-1} & \cdots & a_{1,i_{k}}\beta^{k-1} \\
                            \textbf{U}_{k\times h} & a_{2,i_{h+1}}\beta^{k-2} & a_{2,i_{h+2}}\beta^{k-2} & \cdots & a_{2,i_{k}}\beta^{k-2} \\
                             & \vdots & \vdots & \ddots & \vdots \\
                             & a_{k,i_{h+1}} & a_{k,i_{h+2}} & \cdots & a_{k,i_{k}} \\
                          \end{array}
                        \right)\cdot(\beta^{i_{h+1}-k+1}\beta^{i_{h+2}-k+1}\cdots\beta^{i_{k}-k+1}).$
\end{center}
Clearly, compared with the degree of $\det(\textbf{M}_1)$, the degree of $\det(\textbf{M}_2)$ is less than $m$. Let $\textbf{L}$ be a $(k-h)\times (k-h)$ matrix obtained by removing the $j_1$-th, $j_2$-th, $\ldots$, $j_h$-th rows from the following matrix
\begin{center}
$\left(
\begin{array}{ccc}
              a_{1,i_{h+1}} & \cdots & a_{1,i_{k}} \\
              \vdots & \ddots & \vdots \\
              a_{k,i_{h+1}} & \cdots & a_{k,i_{k}}
            \end{array}
            \right)$.
\end{center}
It is readily checked that the leading coefficient of $\det(\textbf{M}_2)$ is $\det(\textbf{L})$ or $-\det(\textbf{L})$ (this fact comes from two observations: (1) via elementary row-addition operations on $\det(\textbf{M}_2)$, the $\textbf{U}_{k\times h}$ part in $\textbf{M}_2$, which is an upper triangular matrix, can be transformed to a matrix with at most one 1 in each row; (2) $\beta$ has higher degrees in upper rows of $\textbf{M}_2$). Since $\textbf{L}$ is a minor of $\textbf{A}_1,$ $\det(\textbf{L})\neq 0$. So $\det(\textbf{D}_k)\neq 0$.

\textbf{Case $2.$} $\textbf{D}_k$ contains the last column of $\textbf{GB}$. The arguments are similar to those in Case 1.

\textbf{Subcase $2.1.$} $\textbf{D}_k$ doesn't contain any of the first $k$ columns of $\textbf{GB}$. Take
\begin{center}
$\textbf{M}_3=\left(
            \begin{array}{ccccc}
              a_{1,i_1}\beta^{i_1} & a_{1,i_2}\beta^{i_2} & \cdots & a_{1,i_{k-1}}\beta^{i_{k-1}}& a_{1,n-1}\beta^{n} \\
              a_{2,i_1}\beta^{i_1-1} & a_{1,i_2}\beta^{i_2-1} & \cdots & a_{2,i_{k-1}}\beta^{i_{k-1}-1}& a_{2,n-1}\beta^{n-2} \\
              \vdots & \vdots & \ddots & \vdots& \vdots \\
              a_{k,i_1}\beta^{i_1-k+1} & a_{k,i_2}\beta^{i_2-k+1} & \cdots & a_{k,i_{k-1}}\beta^{i_{k-1}-k+1}& a_{k,n-1}\beta^{n-k} \\
            \end{array}
          \right)$,
 \end{center}
where $\{i_{1},i_{2},\ldots,i_{k-1}\}\subseteq\{k,k+1,\ldots,n-2\}$, such that the degree of $\det(\textbf{M}_3)$ is the same as that of $\det(\textbf{D}_k)$, and their leading coefficients are the same. Then
\begin{center}\scriptsize
${\det}(\textbf{M}_3)=\det \left(
            \begin{array}{ccccc}
              a_{1,i_1}\beta^{k-1} & a_{1,i_2}\beta^{k-1} & \cdots & a_{1,i_{k-1}}\beta^{k-1}& a_{1,n-1}\beta^{k} \\
              a_{2,i_1}\beta^{k-2} & a_{2,i_2}\beta^{k-2} & \cdots & a_{2,i_{k-1}}\beta^{k-2}& a_{2,n-1}\beta^{k-2} \\
              \vdots & \vdots & \ddots & \vdots& \vdots \\
              a_{k,i_1} & a_{k,i_2} & \cdots & a_{k,i_{k-1}} & a_{k,n-1} \\
            \end{array}
          \right)\cdot(\beta^{i_1-k+1}\cdots\beta^{i_{k-1}-k+1}\beta^{n-k})$
          $=\det\left(
            \begin{array}{ccccc}
              a_{1,i_1} & a_{1,i_2} & \cdots & a_{1,i_{k-1}}& a_{1,n-1}\beta \\
              a_{2,i_1} & a_{2,i_2} & \cdots & a_{2,i_{k-1}}& a_{2,n-1} \\
              \vdots & \vdots & \ddots & \vdots \\
              a_{k,i_1} & a_{k,i_2} & \cdots & a_{k,i_{k-1}}& a_{k,n-1} \\
            \end{array}
          \right)\cdot(\beta^{k-1}\beta^{k-2}\cdots\beta)\cdot(\beta^{i_1-k+1}\cdots\beta^{i_{k-1}-k+1}\beta^{n-k})
 $,
\end{center}
whose degree is $1+k(k-1)/2+\sum_{j=1}^{k-1} (i_j-k+1)+n-k\leq kn-k^2+1<m$. Since $a_{1,n-1}\in \mathbb F^*_q$ and every $(k-1)$-minor of $\textbf{A}_1$ is nonzero, the leading coefficient of $\det(\textbf{M}_3)$ is nonzero. So $\det(\textbf{D}_k)\neq 0$.

\textbf{Subcase $2.2.$} $\textbf{D}_k$ contains $h$ columns coming from the first $k$ columns of $\textbf{GB}$ for some $1\leq h\leq k$. Write these $h$ columns as the $j_1$-th, $j_2$-th, $\ldots$, $j_h$-th columns. Let $\textbf{U}_{k\times h}$ be the submatrix formed by the first $h$ column of $\textbf{D}_k$. Take
\begin{center}
 $\textbf{M}_4=\left(
                          \begin{array}{c;{2pt/2pt}ccccc}
                             & a_{1,i_{h+1}}\beta^{i_{h+1}}  & \cdots & a_{1,i_{k-1}}\beta^{i_{k-1}} & a_{1,n-1}\beta^{n} \\
                            \textbf{U}_{k\times h} & a_{2,i_{h+1}}\beta^{i_{h+1}-1} & \cdots & a_{2,i_{k-1}}\beta^{i_{k-1}-1} & a_{2,n-1}\beta^{n-2} \\
                             & \vdots &  \ddots & \vdots \\
                             & a_{k,i_{h+1}}\beta^{i_{h+1}-k+1} &  \cdots & a_{k,i_{k-1}}\beta^{i_{k-1}-k+1} & a_{k,n-1}\beta^{n-k} \\
                          \end{array}
                        \right)$,
\end{center}
where $\{i_{h+1},i_{h+2},\ldots,i_{k-1}\}\subseteq\{k,k+1,\ldots,n-2\}$, such that the degree of $\det(\textbf{M}_4)$ is the same as that of $\det(\textbf{D}_k)$, and their leading coefficients are the same. Then
\begin{center}\scriptsize
$\det(\textbf{M}_4)=\det\left(
                          \begin{array}{c;{2pt/2pt}cccc}
                             & a_{1,i_{h+1}}\beta^{k-1} &  \cdots & a_{1,i_{k-1}}\beta^{k-1}& a_{1,n-1}\beta^{k} \\
                            \textbf{U}_{k\times h} & a_{2,i_{h+1}}\beta^{k-2} & \cdots & a_{2,i_{k-1}}\beta^{k-2} & a_{2,n-1}\beta^{k-2} \\
                             & \vdots & \ddots & \vdots& \vdots \\
                             & a_{k,i_{h+1}} & \cdots & a_{k,i_{k-1}}& a_{k,n-1} \\
                          \end{array}
                        \right)\cdot (\beta^{i_{h+1}-k+1}\cdots\beta^{i_{k-1}-k+1}\beta^{n-k})$,
\end{center}
Clearly, compared with the degree of $\det(\textbf{M}_3)$, the degree of $\det(\textbf{M}_4)$ is less than $m$.

\textbf{Subcase $2.2.1.$} $\textbf{D}_k$ contains the first column of $\textbf{GB}$. W.l.o.g., assume that the $j_1$-th column of $\textbf{GB}$ is just its first column. Let $\textbf{L}$ be a $(k-h)\times (k-h)$ matrix obtained by removing the $j_1$-th, $j_2$-th, $\ldots$, $j_h$-th rows from the following matrix
\begin{center}
$\left(
\begin{array}{cccc}
              a_{1,i_{h+1}} & \cdots & a_{1,i_{k-1}} & a_{1,n-1} \\
              \vdots & \ddots & \vdots & \vdots \\
              a_{k,i_{h+1}} & \cdots & a_{k,i_{k-1}} & a_{k,n-1}
            \end{array}
            \right)$.
\end{center}
It is readily checked that the leading coefficient of $\det(\textbf{M}_4)$ is $\det(\textbf{L})$ or $-\det(\textbf{L})$. Since $\textbf{L}$ is a minor of $\textbf{A}_2,$ $\det(\textbf{L})\neq 0$. So $\det(\textbf{D}_k)\neq 0$.

\textbf{Subcase $2.2.2.$} $\textbf{D}_k$ does not contain the first column of $\textbf{GB}$. Let $\textbf{L}$ be a $(k-h-1)\times (k-h-1)$ matrix obtained by removing the first, the $j_1$-th, $j_2$-th, $\ldots$, $j_h$-th rows from the following matrix
\begin{center}
$\left(
\begin{array}{cccc}
              a_{1,i_{h+1}} & \cdots & a_{1,i_{k-1}} & a_{1,n-1} \\
              \vdots & \ddots & \vdots & \vdots \\
              a_{k,i_{h+1}} & \cdots & a_{k,i_{k-1}} & a_{k,n-1}
            \end{array}
            \right)$.
\end{center}
It is readily checked that the leading coefficient of $\det(\textbf{M}_4)$ is $a_{1,n-1}\cdot\det(\textbf{L})$ or $-a_{1,n-1}\cdot\det(\textbf{L})$. Note that $a_{1,n-1}\in \mathbb F^*_q$. Since $\textbf{L}$ is a minor of $\textbf{A}_2,$ $\det(\textbf{L})\neq 0$. So $\det(\textbf{D}_k)\neq 0$. \qed

\section{Appendix}\label{app-B}

\textbf{Proof of Lemma \ref{lem:subcode from Gab}}~~
To construct the required matrix $\textbf{G}$, we first take a $\mathcal{G}[\mu \times (\eta-r), d]_q$ code in vector representation over $\mathbb F_{q^{\mu}} $:

\begin{center}
  $\textbf G_0=\left(
                   \begin{array}{cccc}
                     1 & g_{0,1} & \cdots & g_{0,\eta-r-1} \\
                     1 & g_{0,1}^{[1]} & \cdots & g_{0,\eta-r-1}^{[1]} \\
                     \vdots & \vdots & \ddots & \vdots \\
                     1 & g^{[\kappa-1]}_{0,1} & \cdots & g_{0,\eta-r-1}^{[\kappa-1]} \\
                   \end{array}
                 \right),$
\end{center} where $1, g_{0,1}, ~\ldots, ~g_{0,\eta-r-1}\in \mathbb F_{q^{\mu}} $ are linearly independent over $\mathbb F_q$.

We shall extend $\textbf G_0$ by adding $r$ columns to obtain $\textbf{G}$. We need $r$ steps. For $0\leq i\leq r-1$, in Step $i$, let $\omega_i=\eta-r+i-2$ and $\textbf{G}_i=$
\begin{center}\setlength{\arraycolsep}{3.5pt}
\begin{tiny}
$\left(
\begin{array}{c;{2pt/2pt}ccccccccccccc}
      &      & 0 & 0 & \cdots & 0 & \alpha_{0,\kappa} & \cdots & \alpha_{0,\eta-r-1} & 0 & 0 & \cdots & 0 & 0 \\
      &      & 0 & 0 & \cdots & 0 & \alpha_{1,\kappa} & \cdots & \alpha_{1,\eta-r-1} & \alpha_{1,\eta-r} & 0 & \cdots & 0 & 0 \\
    \textbf{I}_{i}    &    & \vdots & \vdots & \ddots & \vdots & \vdots & \ddots& \vdots & \vdots & \vdots & \ddots &\vdots & \vdots\\
      &      & 0 & 0 & \cdots & 0 & \alpha_{i-2,\kappa} & \cdots & \alpha_{i-2,\eta-r-1} & \alpha_{i-2,\eta-r} & \alpha_{i-2,\eta-r+1} & \cdots & 0 & 0 \\
      &      & 0 & 0 & \cdots & 0 & \alpha_{i-1,\kappa} & \cdots & \alpha_{i-1,\eta-r-1} & \alpha_{i-1,\eta-r} & \alpha_{i-1,\eta-r+1} & \cdots & \alpha_{i-1,\omega_i} & 0 \\ \hdashline[2pt/2pt]
      &      & 1 & g_{i,i+1} & \cdots & g_{i,\kappa-1} & g_{i,\kappa} & \cdots & g_{i,\eta-r-1} & g_{i,\eta-r} & g_{i,\eta-r+1} & \cdots & g_{i,\omega_i} & g_{i,\omega_i+1} \\
      &      & 1 & g^{[1]}_{i,i+1} & \cdots & g^{[1]}_{i,\kappa-1} & g^{[1]}_{i,\kappa} & \cdots & g^{[1]}_{i,\eta-r-1} & g^{[1]}_{i,\eta-r} & g^{[1]}_{i,\eta-r+1} & \cdots & g^{[1]}_{i,\omega_i} & g^{[1]}_{i,\omega_i+1}\\
      &      & \vdots &\vdots & \ddots & \vdots & \vdots & \ddots & \vdots & \vdots & \vdots & \ddots & \vdots & \vdots \\
      &      & 1 & g^{[\kappa-i-1]}_{i,i+1} & \cdots & g^{[\kappa-i-1]}_{i,\kappa-1} & g^{[\kappa-i-1]}_{i,\kappa} & \cdots & g^{[\kappa-i-1]}_{i,\eta-r-1} & g^{[\kappa-i-1]}_{i,\eta-r} & g^{[\kappa-i-1]}_{i,\eta-r+1}  & \cdots & g^{[\kappa-i-1]}_{i,\omega_i} & g^{[\kappa-i-1]}_{i,\omega_i+1} \\
   \end{array}
 \right)$\end{tiny}
\end{center}
be a $\kappa\times (\omega_i+2)$ matrix, where $1, g_{i,i+1},\ldots,g_{i,\omega_i+1} \in \mathbb F_{q^{\mu}}$ are linearly independent over $\mathbb F_q$, and the sub-matrix of $\textbf{G}_i$ obtained by removing its first $i$ rows and the leftmost $i$ columns produces a $\mathcal G[\mu \times (\eta-r),d+i]_q$ code. When $i=0$, $\textbf{G}_i$ is just $\textbf{G}_0$ we defined in the above paragraph. Now, we show that how to obtain  $\textbf{G}_{i+1}$ from  $\textbf{G}_i$ for $0\leq i\leq r-1$.

Let $t_{i, i+1}, t_{i, i+2}, \ldots, t_{i, \kappa-1} \in \mathbb F_{q^\mu}$ such that
\begin{center}\scriptsize
$\textbf{H}_{i,1}$=$\left(
\begin{array}{c;{2pt/2pt}ccccc}
             &     &   &  \\
          \textbf{I}_{i}    &  &  & & \\ \hdashline[2pt/2pt]
              & 1 & t_{i,i+1} & \cdots & t_{i, \kappa-1} \\
                &  & 1  &  & \\
                &   &  & \ddots & \\
                &   &   & & 1\\
         \end{array}
       \right)
       \left(
         \begin{array}{c;{2pt/2pt}ccccc}
                &   &   &  \\
          \textbf{I}_{i}     &  &  & & \\ \hdashline[2pt/2pt]
            &   1 &  &  &  \\
                & -1 & 1  &  & \\
                & \vdots  &  & \ddots & \\
                &  -1 &   & & 1\\
         \end{array}
       \right)\textbf{G}_i$
\end{center}
  \begin{scriptsize}$$
= \left(\begin{array}{c;{2pt/2pt}ccccccccc}
  &    & 0 & 0 & \cdots & 0 & \alpha_{0,\kappa} & \cdots & \alpha_{0,\eta-r-1} \\
        &    & 0 & 0 & \cdots & 0 & \alpha_{1,\kappa} & \cdots & \alpha_{1,\eta-r-1} \\
    \textbf{I}_{i}    &    & \vdots & \vdots & \ddots & \vdots & \vdots & \ddots& \vdots  \\
        &    & 0 & 0 & \cdots & 0 & \alpha_{i-2,\kappa} & \cdots & \alpha_{i-2,\eta-r-1} \\
        &    & 0 & 0 & \cdots & 0 & \alpha_{i-1,\kappa} & \cdots & \alpha_{i-1,\eta-r-1} \\ \hdashline[2pt/2pt]
        &    & 1 & 0 & \cdots & 0 & \alpha_{i,\kappa} & \cdots & \alpha_{i,\eta-r-1}  \\
         &    & 0 & g^{[1]}_{i,i+1}-g_{i,i+1} & \cdots & g^{[1]}_{i,\kappa-1}-g_{i,\kappa-1} & g^{[1]}_{i,\kappa}-g_{i,\kappa} & \cdots & g^{[1]}_{i,\eta-r-1}-g_{i,\eta-r-1}\\
        &    & \vdots &\vdots & \ddots & \vdots & \vdots & \ddots & \vdots  \\
        &    & 0 & g^{[\kappa-i-1]}_{i,i+1}-g_{i,i+1} & \cdots & g^{[\kappa-i-1]}_{i,\kappa-1}-g_{i,\kappa-1} & g^{[\kappa-i-1]}_{i,\kappa}-g_{i,\kappa} & \cdots & g^{[\kappa-i-1]}_{i,\eta-r-1}-g_{i,\eta-r-1} \\
\end{array}\right.$$
$$\left.\begin{array}{cccccc}
      & 0 & 0 & \cdots & 0 & 0 \\
      & \alpha_{1,\eta-r} & 0 & \cdots & 0 & 0 \\
      & \vdots & \vdots & \ddots & \vdots & \vdots \\
      & \alpha_{i-2,\eta-r} & \alpha_{i-2,\eta-r+1} & \cdots & 0 & 0 \\
      & \alpha_{i-1,\eta-r} & \alpha_{i-1,\eta-r+1} & \cdots & \alpha_{i-1,\omega_i} & 0 \\ \hdashline[2pt/2pt]
      & \alpha_{i,\eta-r} & \alpha_{i,\eta-r+1} & \cdots & \alpha_{i,\omega_i} & \alpha_{i,\omega_i+1} \\ 
      & g^{[1]}_{i,\eta-r}-g_{i,\eta-r} & g^{[1]}_{i,\eta-r+1}-g_{i,\eta-r+1} & \cdots & g^{[1]}_{i,\omega_i}-g_{i,\omega_i} & g^{[1]}_{i,\omega_i+1}-g_{i,\omega_i+1}\\
      & \vdots & \vdots & \ddots & \vdots & \vdots \\
      & g^{[\kappa-i-1]}_{i,\eta-r}-g_{i,\eta-r} & g^{[\kappa-i-1]}_{i,\eta-r+1}-g_{i,\eta-r+1} & \cdots & g^{[\kappa-i-1]}_{i,\omega_i}-g_{i,\omega_i} & g^{[\kappa-i-1]}_{i,\omega_i+1}-g_{i,\omega_i+1} \\
   \end{array}\right).$$
\end{scriptsize}

\noindent Notice that $t_{i,i+1},\ldots,t_{i,\kappa-1}$ influence only the first row under the broken line of $\textbf{H}_{i,1}$ and the requirements on this row constitute a linear system of equations with $\kappa-1-i$ equations and $\kappa-1-i$ unknowns. Therefore, the desired $t_{i,i+1},\ldots,t_{i,\kappa-1}$ always exist (this is from the observation of the generator matrix of the Gabidulin code defined by $1, g_{i,i+1},\ldots,g_{i,\kappa-1}$).

Let
\begin{center}\scriptsize
$\textbf{H}_{i,2}$=$\left(
         \begin{array}{c;{2pt/2pt}cccc}
           \textbf{I}_{i+1} &   &   &   &   \\ \hdashline[2pt/2pt]
             & 1 &   &   &   \\
             & -1 & 1 &   &   \\
             &   & \ddots & \ddots &   \\
             &   &   & -1 & 1 \\
         \end{array}
       \right)\textbf{H}_{i,1}$
=\begin{scriptsize}\setlength{\arraycolsep}{2.5pt}$\left(
   \begin{array}{c;{2pt/2pt}cccccccccccc}
             & 0 & \cdots & 0 & \alpha_{0,\kappa} & \cdots & \alpha_{0,\eta-r-1} & 0 & 0 &\cdots & 0 & 0 \\
             & 0 & \cdots & 0 & \alpha_{1,\kappa} & \cdots & \alpha_{1,\eta-r-1} & \alpha_{1,\eta-r} &0& \cdots & 0 & 0 \\
    \textbf{I}_{i+1}       & \vdots & \ddots & \vdots & \vdots & \ddots& \vdots & \vdots & \vdots &\ddots & \vdots & \vdots \\
            & 0 & \cdots & 0 & \alpha_{i-2,\kappa} & \cdots & \alpha_{i-2,\eta-r-1} & \alpha_{i-2,\eta-r} &\alpha_{i-2,\eta-r+1}& \cdots & 0 & 0 \\
            & 0 & \cdots & 0 & \alpha_{i-1,\kappa} & \cdots & \alpha_{i-1,\eta-r-1} & \alpha_{i-1,\eta-r} &\alpha_{i-1,\eta-r+1}& \cdots & \alpha_{i-1,\omega_i} & 0 \\
            & 0 & \cdots & 0 & \alpha_{i,\kappa} & \cdots & \alpha_{i,\eta-r-1} & \alpha_{i,\eta-r} &\alpha_{i,\eta-r+1} &\cdots & \alpha_{i,\omega_i} & \alpha_{i,\omega_i+1} \\ \hdashline[2pt/2pt]
              & f_{i,i+1} & \cdots & f_{i,\kappa-1} & f_{i,\kappa} & \cdots & f_{i,\eta-r-1} & f_{i,\eta-r} & f_{i,\eta-r+1}&\cdots & f_{i,\omega_i} & f_{i,\omega_i+1}\\
              & f_{i,i+1}^{[1]} & \cdots & f_{i,\kappa-1}^{[1]} & f_{i,\kappa}^{[1]} & \cdots & f_{i,\eta-r-1}^{[1]} & f_{i,\eta-r}^{[1]} &f_{i,\eta-r+1}^{[1]}& \cdots & f_{i,\omega_i}^{[1]} & f_{i,\omega_i+1}^{[1]}\\
            &\vdots & \ddots & \vdots & \vdots & \ddots & \vdots & \vdots &\vdots& \ddots & \vdots & \vdots \\
             & f_{i,i+1}^{[\kappa-i-2]} & \cdots & f_{i,\kappa-1}^{[\kappa-i-2]} & f_{i,\kappa}^{[\kappa-i-2]} & \cdots & f_{i,\eta-r-1}^{[\kappa-i-2]} & f_{i,\eta-r}^{[\kappa-i-2]} &f_{i,\eta-r+1}^{[\kappa-i-2]}& \cdots & f_{i,\omega_i}^{[\kappa-i-2]} & f_{i,\omega_i+1}^{[\kappa-i-2]} \\
   \end{array}
 \right),$\end{scriptsize}
\end{center}where $f_{i,j}=g_{i,j}^{[1]}-g_{i,j}$ for $i+1\leq j \leq \omega_i+1$.
For any full-rank matrix $ \textbf{T}_i \in \mathbb F_{q^\mu}^{\kappa \times \kappa}$, the generator matrix $\textbf {T}_i\textbf{G}_i$ defines the same code as $\textbf {G}_i$, so $\textbf{H}_{i,2}$ defines the same code as $\textbf{G}_i$.

We can assert that $f_{i,i+1},f_{i,i+2},  \ldots, f_{i,\omega_i+1} \in \mathbb F_{q^\mu}$ are linearly independent over $\mathbb F_q$. Since $1, g_{i,i+1}, \ldots, g_{i,\omega_i+1} \in \mathbb F_{q^\mu}$ are linearly independent over $\mathbb F_q$, we construct a $\mathcal G[\mu \times (\omega_i-i+2),\omega_i-i+1]_q$ code generated by
\begin{center}
$\left(
   \begin{array}{cccc}
     1 & g_{i,i+1} & \cdots & g_{i,\omega_i+1} \\
     1 & g_{i,i+1}^{[1]} & \cdots & g_{i,\omega_i+1}^{[1]} \\
   \end{array}
 \right)
$.
\end{center}
Since $(0,f_{i,i+1},\ldots,f_{i,\omega_i+1})$ is a codeword of the $\mathcal G[\mu \times (\omega_i-i+2),\omega_i-i+1]_q$ code, then rank$(f_{i,i+1},\ldots,f_{i,\omega_i+1})=\omega_i-i+1$. So, $f_{i,i+1},f_{i,i+2}, \ldots,f_{i,\omega_i+1}$ are linearly independent over $\mathbb F_q$.

Additionally, since $\mu\geq \eta-r=\omega_i-i+2$, there exists an element $f_{i,\omega_i+2} \in \mathbb F_{q^{\mu}}$ which is $\mathbb F_q$-linearly independent of $f_{i,i+1},\ldots,f_{i,\omega_i+1}$. Hence, the $\kappa \times (\omega_i+3)$ matrix
\begin{center}
$ \textbf{H}_{i,3}=\left(
    \begin{array}{c;{2pt/2pt}cc}
       &  \textbf{0}_{(i+1)\times 1} \\
       &   f_{i,\omega_i+2} \\
      \textbf{H}_{i,2} &   f_{i,\omega_i+2}^{[1]} \\
       &   \vdots \\
       &   f_{i,\omega_i+2}^{[\kappa-i-2]} \\
    \end{array}
  \right)
$
\end{center}
defines with its right bottom $(\kappa-i-1) \times (\omega_i-i+2)$ submatrix a $\mathcal{G}[\mu \times (\omega_i-i+2), d+i+1]_q$ code.

Now we set
\begin{center}\scriptsize
  $\textbf{G}_{i+1}=\left(
                                            \begin{array}{c;{2pt/2pt}cccc}
                                             &   &   &  &  \\
                                              \textbf{I}_{i+1} &   &   &   &  \\ \hdashline[2pt/2pt]
                                                & f_{i,i+1}^{-1} &  &  &  \\
                                                &   & f_{i,i+1}^{-[1]} &  &  \\
                                                &   &   & \ddots &  \\
                                                &   &   &   & f_{i,i+1}^{-[\kappa-i-2]} \\
                                            \end{array}
                                          \right)\textbf{H}_{i,3}$
\begin{tiny}\setlength{\arraycolsep}{2pt}$=\left(
   \begin{array}{c;{2pt/2pt}cccccccccccccc}
&   0 & 0 &\cdots & 0 & \alpha_{0,\kappa} & \cdots & \alpha_{0,\eta-r-1} & 0 & 0& \cdots & 0 & 0 \\
& 0 & 0 &\cdots & 0 & \alpha_{1,\kappa} & \cdots & \alpha_{1,\eta-r-1} & \alpha_{1,\eta-r} & 0& \cdots & 0 & 0 \\
    \textbf{I}_{i+1}        & \vdots &\vdots & \ddots & \vdots & \vdots & \ddots& \vdots & \vdots & \vdots & \ddots& \vdots & \vdots \\
    & 0 &0 & \cdots & 0 & \alpha_{i-1,\kappa} & \cdots & \alpha_{i-1,\eta-r-1} & \alpha_{i-1,\eta-r} & \alpha_{i-1,\eta-r+1} & \cdots& 0 & 0 \\
    & 0 &0 & \cdots & 0 & \alpha_{i,\kappa} & \cdots & \alpha_{i,\eta-r-1} & \alpha_{i,\eta-r} &  \alpha_{i,\eta-r+1} & \cdots & \alpha_{i,\omega_{i+1}} & 0 \\ \hdashline[2pt/2pt]
 & 1 &g_{i+1,i+2} & \cdots & g_{i+1,\kappa-1} & g_{i+1\kappa} & \cdots & g_{i+1,\eta-r-1} & g_{i+1,\eta-r}& g_{i+1,\eta-r+1} &\cdots &g_{i+1,\omega_{i+1}}  & g_{i+1,\omega_{i+1}+1} \\
   & 1 &g_{i+1,i+2}^{[1]} & \cdots & g_{i+1,\kappa-1}^{[1]} & g_{i+1\kappa}^{[1]} & \cdots & g_{i+1,\eta-r-1}^{[1]} & g_{i+1,\eta-r}^{[1]} & g_{i+1,\eta-r+1}^{[1]} & \cdots&g_{i+1,\omega_{i+1}}^{[1]} & g_{i+1,\omega_{i+1}+1}^{[1]} \\
   &\vdots & \vdots &\ddots & \vdots & \vdots & \ddots & \vdots & \vdots & \vdots & \ddots &\vdots & \vdots \\
    & 1 & g_{i+1,i+2}^{[\kappa-i-2]} & \cdots & g_{i+1,\kappa-1}^{[\kappa-i-2]} & g_{i+1\kappa}^{[\kappa-i-2]} & \cdots & g_{i+1,\eta-r-1}^{[\kappa-i-2]} & g_{i+1,\eta-r}^{[\kappa-i-2]} & g_{i+1,\eta-r+1}^{[\kappa-i-2]}  &\cdots & g_{i+1,\omega_{i+1}}^{[\kappa-i-2]} & g_{i+1,\omega_{i+1}+1}^{[\kappa-i-2]} \\
   \end{array}
 \right),$
 \end{tiny}
  \end{center}
where $\omega_{i+1}=\omega_i+1$ and $g_{i+1,j}=f_{i,j}f_{i,i+1}^{-1}$ for $j\in\{i+2,\ldots,\omega_{i+1}+1\}$.
Notice that $1,g_{i+1,i+2}, \ldots, g_{i+1,\omega_{i+1}+1}$ are linearly independent over $\mathbb F_q$, and the right bottom $(\kappa-i-1) \times (\omega_{i+1}-i+1)$ submatrix of $\textbf{G}_{i+1}$ can produce the same $\mathcal{G}[\mu \times (\omega_i-i+2), d+i+1]_q$ code as the one produced by $\textbf{H}_{i,3}$.

Finally, we can choose an invertible matrix $\textbf{T} \in \mathbb F_{q^\mu}^{(\kappa-r) \times (\kappa-r)}$ such that
\begin{center}
 $\textbf{G}=\left(
                                    \begin{array}{cc}
                                     \textbf{I}_{r \times r} &  \\
                                      & \textbf{T} \\
                                    \end{array}
                                  \right) \cdot \textbf{G}_r$
\end{center}
is our required matrix. \qed

\subsection*{Acknowledgements}
The authors thank the anonymous referees for their valuable comments and suggestions that helped improve the equality of the paper.

\end{document}